\documentclass{amsart}
\usepackage{amssymb,amsthm,tikz}
\usetikzlibrary{arrows,shapes,positioning,decorations.markings}
\newtheorem{theorem}{Theorem}
\newtheorem{lemma}[theorem]{Lemma}
\newtheorem{definition}[theorem]{Definition}
\newtheorem{corollary}[theorem]{Corollary}

\newtheorem{problem}[theorem]{Problem}
\newtheorem{example}[theorem]{Example}
\newtheorem{notation}[theorem]{Notation}
\newtheorem{remark}[theorem]{Remark}
\newtheorem{proposition}[theorem]{Proposition}
\newcommand{\Aut}{\mathop{\mathrm{Aut}}}

\newcommand{\Cay}{\mathop{\mathrm{Cay}}}
\newcommand{\Sym}{\mathrm{Sym}}

\newcommand{\SL}{\mathop{\textrm{SL}}}

\renewcommand{\wr}{\mathop{\mathrm{wr}}}

\newcommand{\soc}{\mathop{\mathrm{Soc}}}
\newcommand{\Diag}{\mathop{\mathrm{Diag}}}

\begin{document}
\title[Bounding vertex-stabiliser]{
Bounding the size of a vertex-stabiliser in a finite vertex-transitive graph
} 

\author[C. E. Praeger]{Cheryl E. Praeger}
\address{Cheryl E. Praeger,  School of Mathematics and Statistics,\newline
The University of Western Australia,
 Crawley, WA 6009, Australia} \email{praeger@maths.uwa.edu.au}
\author[P. Spiga]{Pablo Spiga}
\address{Pablo Spiga,  School of Mathematics and Statistics,\newline
The University of Western Australia,
 Crawley, WA 6009, Australia} \email{spiga@maths.uwa.edu.au}
\author[G. Verret]{Gabriel Verret}
\address{Gabriel Verret, Institute of Mathematics, Physics, and
  Mechanics, \newline 
Jadranska 19, 1000 Ljubljana, Slovenia}
\email{gabriel.verret@fmf.uni-lj.si}

\thanks{Address correspondence to P. Spiga,
E-mail: spiga@maths.uwa.edu.au\\ 
The paper forms part of the Australian Research Council Federation
Fellowship Project 
FF0776186 of the first author. The second author is supported by The University of Western Australia
as part of the  Federation Fellowship project.}


\begin{abstract}
In this paper we discuss a method for bounding the size of the stabiliser of a vertex in a $G$-vertex-transitive graph $\Gamma$. In the main result the group $G$ is  quasiprimitive or biquasiprimitive on the vertices of $\Gamma$, and  we obtain a genuine reduction to the case where $G$ is a nonabelian simple group. 

Using normal quotient techniques developed by the first author, the main theorem applies to general $G$-vertex-transitive graphs which are $G$-locally primitive (respectively, $G$-locally quasiprimitive), that is, the stabiliser $G_\alpha$ of a vertex $\alpha$ acts primitively (respectively quasiprimitively) on the set of vertices adjacent to $\alpha$.   We discuss how our results may be used to investigate conjectures by Richard Weiss (in 1978) and the first author (in 1998) that  the order of $G_\alpha$ is bounded above by some function depending only on the valency of $\Gamma$, when $\Gamma$ is $G$-locally primitive or $G$-locally quasiprimitive, respectively.

\bigskip
\begin{center}
\emph{To Richard Weiss for the inspiration of his beautiful conjecture}\\
\end{center}
\end{abstract}

\maketitle

\section{Introduction}\label{sec:1}

In this paper we study the family $\mathcal{A}(d)$ defined as follows (where a graph $\Gamma$ is $G$-\emph{vertex-transitive} if $G$ is a subgroup of $\Aut(\Gamma)$ acting
transitively on the vertex set $V\Gamma$ of $\Gamma$).

\begin{definition}\label{def:Ad}{\rm Let $d$ be a positive integer. The family $\mathcal{A}(d)$ consists of the ordered pairs $(\Gamma,G)$, with $\Gamma$ a connected $G$-vertex-transitive graph of valency at most $d$. 
}
\end{definition}

We study the order of the stabilisers $G_\alpha$ for pairs $(\Gamma,G)\in\mathcal{A}(d)$ and $\alpha\in V\Gamma$, focusing on the case where $G$ is quasiprimitive or biquasiprimitive in its action on $V\Gamma$. A permutation group $G$ is said to be \emph{quasiprimitive} if every non-identity normal subgroup of $G$ is transitive, and \emph{biquasiprimitive} if $G$ is not quasiprimitive and every non-trivial normal subgroup of $G$ has at most two orbits.

We briefly explain the context: the family $\mathcal{A}(d)$ is closed under forming normal quotients in the sense that, for $(\Gamma,G)\in\mathcal{A}(d)$, and a normal subgroup $N\leq G$ with at least three orbits in $V\Gamma$, the pair $(\Gamma_N,G_N)\in\mathcal{A}(d)$, where $\Gamma_N$ is the $G$-normal quotient of $\Gamma$ modulo $N$ (see Definition~\ref{def:nq}) and $G_N$ is the group induced by $G$ on the set of $N$-orbits. We regard pairs $(\Gamma,G)$ which admit no proper normal quotient reduction to smaller graphs as `irreducible'. Thus the irreducible pairs in $\mathcal{A}(d)$, are those for which  $G$ is quasiprimitive or biquasiprimitive on $V\Gamma$.  

 Before stating our main results we give two definitions.

\begin{definition}\label{soclefactor}{\rm
If $G$ is a quasiprimitive or a biquasiprimitive permutation group, then the socle $\soc(G)$ of $G$ is isomorphic to a direct product $T^l$ of isomorphic simple groups (where $T$ is possibly abelian). We call $T$ the \emph{socle factor} of $G$.}
\end{definition}

\begin{definition}\label{functions}{\rm
Let $f:\mathbb{N}\to \mathbb{N}$. Then $(\Gamma,G)\in \mathcal{A}(d)$ is said to be \emph{$f$-bounded} if $|G_\alpha|\leq f(d)$ for every $\alpha\in V\Gamma$.  
Define $\widehat{f}, \tilde{f}:\mathbb{N}\to\mathbb{N}$  by
\[
\begin{array}{lll}
\widehat{f}(d)&=&(df(d)^{d^df(d)^{2d}})!  \\
\tilde{f}(d)&=&f(d')\\
 \end{array}
\]
where  $d'$ is the unique element of $\mathbb{N}$ such that $(d'-1)(d'-2)<d\leq d'(d'-1)$.
For functions $g_1,g_2:\mathbb{N}\to \mathbb{N}$ and setting $d_0=d(d-1)$,  define $\overline{g_1\ast g_2}:\mathbb{N}\to \mathbb{N}$  by 
\[
\overline{g_1\ast g_2}(d)=(d_0(g_1(d_0)g_2(d_0))^{d_0^{d_0}\min\{g_1(d_0),g_2(d_0)\}^{2d_0}})!
\]
} 
\end{definition}

Theorems~\ref{thm:mainqp} and~\ref{thm:mainbiqp} are our main results for quasiprimitive groups and biquasiprimitive groups, respectively.

\begin{theorem}\label{thm:mainqp}
Let $(\Gamma,G)\in \mathcal{A}(d)$ where  $G$ is quasiprimitive on $V\Gamma$ with socle factor $T$. Then either
\begin{description}
\item[$(1)$]$(\Gamma,G)$ is $(dd!)!$-bounded, or
\item[$(2)$]the pair $(\Gamma,G)$ uniquely determines a pair $(\Lambda,T)\in \mathcal{A}(d)$, and if  $(\Lambda,T)$ is $g$-bounded for some $g:\mathbb{N}\to \mathbb{N}$, then $(\Gamma,G)$ is $\widehat{g}$-bounded. Conversely, if $(\Gamma,G)$ is $f$-bounded for some $f:\mathbb{N}\to\mathbb{N}$, then also $(\Lambda,T)$ is $f$-bounded. 
\end{description}
\end{theorem}

\begin{theorem}\label{thm:mainbiqp}
Let $(\Gamma,G)\in \mathcal{A}(d)$ where  $G$ is biquasiprimitive on $V\Gamma$ with socle factor $T$. Then either
\begin{description}
\item[$(1)$]$(\Gamma,G)$ is $(d^2((d(d-1))!)^2)!$-bounded, or
\item[$(2)$]$(\Gamma,G)$ uniquely determines two (possibly isomorphic) pairs $(\Lambda_i,T)\in$\\ $\mathcal{A}(d(d-1))$, for $i=1,2$, and  if  $(\Lambda_i,T)$ is $g_i$-bounded for $i=1,2$, then $(\Gamma,G)$ is $\overline{g_1\ast g_2}$-bounded. Conversely, if $(\Gamma,G)$ is $f$-bounded for some $f:\mathbb{N}\to\mathbb{N}$, then each of the $(\Lambda_i,T)$ is $\tilde{f}$-bounded. 
\end{description}
\end{theorem}

The class of quasiprimitive permutation groups may be described 
(see~\cite{P1}) in a fashion very similar to the description given by he
O'Nan-Scott Theorem for primitive permutation groups. In~\cite{P2} this description is refined and  eight
types of quasiprimitive groups are defined, namely HA, HS, HC, SD, CD, TW,
PA and AS, such that every quasiprimitive group belongs to
exactly one of these types. In proving Theorem~\ref{thm:mainqp} we show, in Corollaries~\ref{cor:many} and~\ref{cor:SD},  that there exists a function $h:\mathbb{N}\to\mathbb{N}$ such that, if $(\Gamma,G)$ is in $\mathcal{A}(d)$, with $G$ quasiprimitive on vertices of type HA, HS, HC, SD, CD or TW, then $(\Gamma,G)$ is $h$-bounded. Furthermore, if $G$ is of type PA with socle $T^l$, then $l$ is bounded above by a function of $d$ and of the size of the vertex-stabiliser of a $T$-vertex-transitive graph $\Lambda$ uniquely determined by $(\Gamma,G)$. Therefore Theorem~\ref{thm:mainqp}  reduces the problem of bounding (as a function of $d$) the size of the vertex-stabiliser $G_\alpha$ to the case of nonabelian simple groups $G$. Similarly, the class of biquasiprimitive groups is described in detail in~\cite{P1b} and Theorem~\ref{thm:mainbiqp} applies to this more complicated family.

The novelty of these (to us, remarkable) results lies in the fact that Theorems~\ref{thm:mainqp} and~\ref{thm:mainbiqp} do not require any assumption on the \emph{local action}, that is, on the action of $G_\alpha$ on the set $\Gamma(\alpha)$ of vertices adjacent to $\alpha$. In particular, $G_\alpha$ is not assumed to be transitive on $\Gamma(\alpha)$ in Theorems~\ref{thm:mainqp} and~\ref{thm:mainbiqp}. 

In the remainder of this introductory section we discuss how Theorems~\ref{thm:mainqp} and~\ref{thm:mainbiqp} can be used to study general $G$-vertex-transitive graphs. Also we show that Theorems~\ref{thm:mainqp} and~\ref{thm:mainbiqp} are relevant for some open problems in algebraic graph theory.

\subsection{Application of Theorems~\ref{thm:mainqp} and~\ref{thm:mainbiqp} to studying general $G$-vertex-transitive graphs}

Although Theorems~\ref{thm:mainqp} and~\ref{thm:mainbiqp}  are stated for quasiprimitive and biquasiprimitive groups, using normal quotient techniques they can be fruitfully applied in more general situations.

\begin{definition}\label{def:nq}{\rm
Let $(\Gamma,G)$ be in $\mathcal{A}(d)$, and let $N$ be a normal subgroup of $G$ with $N$ intransitive on $V\Gamma$. Let $\alpha^N$ denote the $N$-orbit containing $\alpha\in  V\Gamma$. Then the \emph{normal quotient} $\Gamma_N$ is the graph whose vertices are the $N$-orbits on $V\Gamma$, with an edge between distinct vertices $\alpha^N$ and $\beta^N$ if
and only if there is an edge of $\Gamma$ between $\alpha'$ and $\beta'$, for some $\alpha' \in \alpha^N$ and some $\beta'\in \beta^N$. The normal quotient is non-trivial if $N\ne1$.
}
\end{definition} 

Note that the group $G$ induces a transitive action on the vertices of the normal quotient $\Gamma_N$. Also, for adjacent $\alpha^N, \beta^N$ of $\Gamma_N$, each vertex of $\alpha^N$ is adjacent to at least one vertex of $\beta^N$ (since $N$ is transitive on both sets) and so, if $d$ is the valency of $\Gamma$ and  $d'$ is the valency of $\Gamma_N$, then $d'\leq d$. Thus $(\Gamma,G)\in\mathcal{A}(d)$ implies that $(\Gamma_N,G_N)\in\mathcal{A}(d)$, where $G_N$ is the permutation group induced by $G$ on the set of $N$-orbits.

Following Wielandt~\cite{W1}, for a subgroup $H$ of a permutation group $G$,  the $1$-\emph{closure of $H$ in  $G$} is the largest subgroup of $G$ with the same orbits as $H$. The subgroup $H$ is $1$-\emph{closed in} $G$ if $H$ equals its $1$-closure in $G$. Note that if $N$ is a normal subgroup of $G$, then the $1$-closure of $N$ in $G$ is normal in $G$ and equals the kernel of the action of $G$ on the $N$-orbits. Thus the induced group $G_N$ is the quotient of $G$ modulo the 1-closure of $N$ in $G$. In particular, if $N$ is a $1$-closed normal subgroup of $G$, then $G_N= G/N$.

Let $(\Gamma,G)\in \mathcal{A}(d)$ and let $N$ be a normal subgroup which is $1$-closed in $G$, and is maximal subject to having more than two orbits on $V\Gamma$. By Definition~\ref{def:nq}, the pair $(\Gamma_N,G/N)$ lies in $\mathcal{A}(d)$ and the group $G/N$ is quasiprimitive or biquasiprimitive on $V\Gamma_N$. Hence Theorems~\ref{thm:mainqp} and~\ref{thm:mainbiqp} apply to $(\Gamma_N,G/N)$. Now the stabiliser of the vertex $\alpha^N$ of $\Gamma_N$ is $G_\alpha N/N\cong G_\alpha/N_\alpha$ and therefore Theorems~\ref{thm:mainqp} and~\ref{thm:mainbiqp} provide information about bounds on $|G_\alpha/N_\alpha|$. In general, without further restrictions on $(\Gamma,G)$, it is difficult to obtain useful information about $N_\alpha$. In general, $|N_\alpha|$ is not bounded by a function of $d$ (see Example~\ref{ex:1} which gives a family of examples in $\mathcal{A}(d)$ for which $|N_\alpha|$ grows exponentially with $d$). Nevertheless, there are some remarkable families of $G$-vertex-transitive graphs where fairly weak conditions on the local action lead to an upper bound on $|N_\alpha|$ as a function of $d$. We discuss in detail some of these families in the rest of this subsection. 

For a property $\mathcal{P}$ of a group action, a pair $(\Gamma,G)\in\mathcal{A}(d)$ is said to be  \emph{locally $\mathcal{P}$} if the permutation group $G_\alpha^{\Gamma(\alpha)}$ induced by $G_\alpha$ on $\Gamma(\alpha)$ has the property $\mathcal{P}$. We will consider four properties $\mathcal{P}$, the first three of which are the properties of being $2$-transitive, primitive or quasiprimitive. The fourth property is semiprimitivity: a finite permutation group $L$ is \emph{semiprimitive} if every normal subgroup of $L$ is either transitive or semiregular~\cite{BM,KS,PSV}.


The following proposition was proved in~\cite[Lemmas 1.1, 1.4(p), 1.5 and 1.6]{ImpGr} (see the summary in \cite[Theorem 4.1]{PConj}, or~\cite{Montreal} for a more general treatment). The boundedness assertion in the last sentence follows since, as noted above, the vertex stabilisers for $(\Gamma,G)$ and $(\Gamma_N,G/N)$ are isomorphic. 

\begin{proposition}\label{prop}
Let $\mathcal{P}$ be one of the properties: $2$-transitive, primitive or quasiprimitive. Let $(\Gamma,G)\in \mathcal{A}(d)$ be locally $\mathcal{P}$, and let $N$ be a normal subgroup which is $1$-closed in $G$, and maximal subject to having more than two orbits on $V\Gamma$. Then $(\Gamma_N,G/N)\in\mathcal{A}(d)$ is locally $\mathcal{P}$, $G/N$ is quasiprimitive or biquasiprimitive on $V\Gamma_N$, and  $N_\alpha=1$ for every $\alpha\in V\Gamma$. In particular, for any function $f:\mathbb{N}\to \mathbb{N}$,   $(\Gamma,G)$ is $f$-bounded if and only if $(\Gamma_N,G/N)$ is $f$-bounded.
\end{proposition}

Proposition~\ref{prop} shows that, for pairs $(\Gamma,G)$ which are locally $2$-transitive, locally primitive, or locally quasiprimitive, normal quotient reduction together with Theorems~\ref{thm:mainqp} and~\ref{thm:mainbiqp} can be used to obtain useful information about the vertex-stabiliser $G_\alpha$.

\subsection{The Weiss and Praeger conjectures}
In 1973, Gardiner~\cite{Gard} proved that, if $\Gamma$ is a connected $G$-vertex-transitive locally primitive graph and $(u,v)$ is an arc of $\Gamma$, then the pointwise stabiliser in $G$ of $u, v$, and all vertices adjacent to either $u$ or $v$, is a $p$-group, for some prime $p$. A series of papers (see~\cite{weissp,weissu} for example) followed in which various additional constraints on the local action led to upper bounds on the order of a vertex-stabiliser. This eventually led Richard Weiss~\cite{Weiss} to conjecture in 1978 that local primitivity should imply boundedness. In our terminology his conjecture is the following.\\

\noindent\textbf{Weiss Conjecture. }\label{Weiss}
\emph{There exists a function $f:\mathbb{N}\to \mathbb{N}$ such that, if $(\Gamma,G)\in \mathcal{A}(d)$  is locally primitive, then $(\Gamma,G)$ is $f$-bounded.}\\

In~1998, the first author~\cite[Problem 7]{PConj} conjectured that local primitivity can be weakened to local quasiprimitivity.\\

\noindent\textbf{Praeger Conjecture. }\label{Cheryl}
\emph{There exists a function $f:\mathbb{N}\to \mathbb{N}$ such that, if $(\Gamma,G)\in \mathcal{A}(d)$  is locally quasiprimitive, then $(\Gamma,G)$ is $f$-bounded.}\\

Moreover, in~\cite{PSV} the local assumption was weakened further to semiprimitivity (defined above). \\

\noindent\textbf{PSV Conjecture. }\label{sp}
There exists a function $f:\mathbb{N}\to \mathbb{N}$ such that, if $(\Gamma,G)\in \mathcal{A}(d)$  is locally semiprimitive, then $(\Gamma,G)$ is $f$-bounded.
\\

In spirit, these conjectures are similar to the~1967
conjecture of Charles Sims~\cite{Sims}, proved in~\cite{CPSS}, that for a
$G$-vertex-primitive graph $\Gamma$, the order of the stabiliser of a
vertex is bounded above by some function of the
valency of $\Gamma$. Unfortunately the methods in~\cite{CPSS}, using information about maximal subgroups of nonabelian simple groups, are not transferable to attack the other conjectures, and all three remain open.

As we hinted to above, one approach towards proving the Weiss Conjecture was to prove subcases of it by placing additional constraints on the local action. After a series of papers by Trofimov~\cite{Tr1,Tr2,Tr3,Tr4}, this approach culminated in 1994 in a proof of the Weiss Conjecture in the case of locally $2$-transitive graphs. 
Also some progress on the PSV Conjecture was made in~\cite{PSV}, by restricting further the local semiprimitive action. 

An alternative approach to studying the Weiss Conjecture was initiated by the first author in \cite{ImpGr,PConj} using normal quotients  to reduce to quasiprimitive and biquasiprimitive vertex actions. This was taken further in~\cite{CLP}, where an analysis of $G$-locally primitive graphs with $G$ quasiprimitive on vertices was undertaken, considering separately each of the eight types of quasiprimitive groups according to the quasiprimitive group subdivision described in~\cite{P2}. 
For six of the eight quasiprimitive types it was proved that $|G_\alpha|$ is bounded above by an explicit function of the valency, reducing the problem of proving the Weiss conjecture for quasiprimitive group actions to the almost simple and product action types AS and PA (\cite[Section~$2$]{CLP}). The PA type was also examined in~\cite[Proposition~$2.2$]{CLP} but, unfortunately, the proof contains an error which we discovered while working on the Praeger conjecture. (Example~\ref{ex:4} 
gives a counter-example to a claim made in the proof of~\cite[Proposition~$2.2$]{CLP}.)

Our results shed new light on the Weiss and Praeger Conjectures. Indeed, Proposition~\ref{prop} together with  Theorems~\ref{thm:mainqp} and~\ref{thm:mainbiqp} show that the Weiss and Praeger Conjectures hold true if and only if the graphs 
$(\Lambda,T)$ in Theorem~\ref{thm:mainqp} and  $(\Lambda_i,T)$ in Theorem~\ref{thm:mainbiqp} are $f$-bounded for some function $f$ depending only on their valencies. Thus 
Theorems~\ref{thm:mainqp} and~\ref{thm:mainbiqp} reduce the Weiss and Praeger Conjectures to a sequence of problems about $T$-vertex-transitive graphs, for various families of non-abelian simple groups $T$ with $|T|\to\infty$. These problems are discussed and successfully solved for many families of simple groups in~\cite{WEISS}. 

Although $(\Lambda,T)$  in Theorem~\ref{thm:mainqp}, and the $(\Lambda_i,T)$ in Theorem~\ref{thm:mainbiqp}, are uniquely determined by $(\Gamma,G)$ and inherit many of the properties of $(\Gamma,G)$, we will see in Examples~\ref{ex:2},~\ref{ex:3} and~\ref{ex:4} that `local properties' are not necessarily preserved by this reduction (for instance, if $(\Gamma,G)$ is locally primitive, then $(\Lambda,T)$ is not 
 necessarily even locally quasiprimitive).  

It would therefore be very interesting to find which local properties of $(\Gamma,G)$ are inherited by the $(\Lambda_i,T)$. In fact, it may be possible to prove the Weiss and the Praeger Conjectures, using Theorems~\ref{thm:mainqp} and~\ref{thm:mainbiqp}, by proving a stronger conjecture for $(\Gamma,T)\in \mathcal{A}(d)$, with $T$ in a family of  non-abelian simple groups, in which the local action of $(\Gamma,T)$ is further relaxed. We leave this as an open problem.

\begin{problem}\label{prob}Which properties of $(\Gamma,G)\in \mathcal{A}(d)$, with $G$ quasiprimitive on vertices, are inherited by $(\Lambda,T)$ in Theorem~\ref{thm:mainqp}? 
Which properties of $(\Gamma,G)\in \mathcal{A}(d)$, with $G$ biquasiprimitive  on vertices, are inherited by $(\Lambda_1,T),(\Lambda_2,T)$ in Theorem~\ref{thm:mainbiqp}?
\end{problem}

\subsection{Structure of the paper}

In Section~\ref{sec:key} we give some preliminary and fundamental results that are of importance in the rest of our work. Before stating the main theorem of Section~\ref{sec:key} in its full generality, we need a definition extending Definition~\ref{def:Ad}.

\begin{definition}\label{Nbdd}
{\rm Let $f_1,f_2:\mathbb{N}\to \mathbb{N}$ be functions. Let $\Gamma$ be a connected graph with every vertex of valency at most $d$ and $N\leq \Aut(\Gamma)$. We say that $(\Gamma,N)$ is \emph{$(f_1,f_2)$-bounded} if the number of orbits of $N$ on $V\Gamma$ is at most $f_1(d)$, and also $|N_\alpha|\leq f_2(d)$ for every $\alpha\in V\Gamma$.} 
\end{definition}

For example, if $(\Gamma,G)\in\mathcal{A}(d)$ is $f$-bounded, then $(\Gamma,G)$ is $(1,f)$-bounded where $1$ denotes the constant function with value 1.
In Section~\ref{sec:key}, we prove the following result which often leads to reductions in  `boundedness proofs'.
\begin{theorem}\label{thm:1}
Let $(\Gamma,N)$ be $(f_1,f_2)$-bounded and $G\leq \Aut(\Gamma)$ with $N\trianglelefteq G$. Then there exists a function $f_3:\mathbb{N}\to \mathbb{N}$  such that $(\Gamma,G)$ is $(f_1,f_3)$-bounded.
\end{theorem}

\begin{remark}\label{rmrm}
{\rm The proof of Theorem~\ref{thm:1} is constructive and it shows that we can take $f_3(d)=d^{f_1(d)-1}(df_1(d)^2f_2(d))!$.  Here we do not claim that such an $f_3$ is the best possible function for Theorem~\ref{thm:1}. But it would be interesting to know whether our choice of $f_3$ could be \emph{significantly} improved.}
\end{remark}

Section~\ref{sec:auxiliary} contains an auxiliary lemma needed in the proof of Theorem~\ref{thm:mainqp}. Section~\ref{redToAS} contains the proof of Theorem~\ref{thm:mainqp}  and Section~\ref{redToAS2} contains the proof of Theorem~\ref{thm:mainbiqp}. 
Finally, Section~\ref{sec:examples} contains the examples mentioned in the introduction.

In this paper all groups will be finite and  graphs will be finite and
simple. 

\section{Boundedness}\label{sec:key}

We start this section recalling some standard definitions. Given a group $N$ and a subset $S$ of $N$, we define the Cayley digraph of $N$ over $S$ (denoted by $\Cay(N,S)$) as the digraph with vertex set $N$ and with arc set $\{(n,n')\in N\times N\mid n(n')^{-1}\in S\}$. Clearly, $\Cay(N,S)$ is an undirected graph if and only if $S=S^{-1}$. Also the number of connected components of $\Cay(N,S)$ is $|N:\langle S\rangle|$. The main aim in this section is to prove Theorem~\ref{thm:1} for the function $f_3(d)$  given in Remark~\ref{rmrm}.

\medskip
\noindent\emph{Proof of Theorem~\ref{thm:1}. }
Let $(\Gamma,N)$ be a connected $(f_1,f_2)$-bounded graph and $N\trianglelefteq G\leq \Aut(\Gamma)$. 
Since $N\subseteq G$ and $N$ has at most $f_1(d)$ orbits, the group $G$ also has at most $f_1(d)$ orbits. It remains to show that $|G_\alpha|\leq f_3(d)$ for every $\alpha\in V\Gamma$.
 
Let $\mathcal{O}_1,\ldots,\mathcal{O}_t$ be the orbits of $N$ on $V\Gamma$. We claim that $\Gamma$ contains $t$ vertices $\beta_1,\ldots,\beta_t$, with $\beta_i\in \mathcal{O}_i$ for every $i$, such that the subgraph induced by $\Gamma$ on $\{\beta_1,\ldots,\beta_t\}$ is connected. Let $X$ be a subset of vertices of $\Gamma$ of maximal size with the properties   $|X\cap \mathcal{O}_i|\leq 1$ for every $i$, and the subgraph induced by $\Gamma$ on $X$ is connected. If $|X|=t$, then the claim is proved. Suppose then that $|X|=l<t$. Without loss of generality we may assume that $X=\{\beta_1,\ldots,\beta_l\}$. Let $v$ be a vertex in $\mathcal{O}_{l+1}$. Since $\Gamma$ is connected, there exists a path $\beta_1=v_1,\ldots,v_u=v$ in $\Gamma$ from $\beta_1$ to $v$. Let $i$ be minimal such that $v_i\notin \cup_{j=1}^l\mathcal{O}_j$. In particular, $i\geq 2$ and $v_{i-1}\in \mathcal{O}_{k}$ for some $k\leq l$. Since $N$ is transitive on $\mathcal{O}_{k}$, there exists $n\in N$ such that $\beta_{k}=v_{i-1}^n$. So, $v_i^n$ is adjacent to $\beta_{k}$ and $v_i^n\notin \cup_{j=1}^l\mathcal{O}_j$. Set $X'=X\cup\{v_i^n\}$. By construction, $X\subset X'$, the subgraph induced by $\Gamma$ on $X'$  is connected and $X'$ contains at most one vertex from each $\mathcal{O}_i$. This contradicts the maximality of $X$. Thus $|X|=t$ and the claim is proved. 

Fix $\beta_1,\ldots,\beta_t$, with $\beta_i\in \mathcal{O}_i$ for every $i$, such that the subgraph induced by $\Gamma$ on $\{\beta_1,\ldots,\beta_t\}$ is connected. Let $S$ be the set 
\begin{equation*}
S=\{n\in N\mid \textrm{there exists }i \textrm{ with }\beta_i^n\in\cup_{j=1}^t\Gamma(\beta_j)\}
\end{equation*} and let $\tilde{\Gamma}$ be the Cayley digraph on $N$ with connection set
$S$, that is, $\tilde{\Gamma}=\mathrm{Cay}(N,S)$. Given $1\leq i,j\leq t$, the number of elements $n\in N$ with $\beta_i^n\in \Gamma(\beta_j)$ is at most $|\Gamma(\beta_j)||N_{\beta_i}|\leq df_2(d)$. Therefore $|S|\leq t^2df_2(d)\leq df_1(d)^2f_2(d)$.

Set $H=\cap_{j=1}^tG_{\beta_j}$ and note that, by connectivity of the induced graph on $\{\beta_1,\ldots,\beta_t\}$, for each $i$, the index $|G_{\beta_i}:H|\leq d(d-1)^{t-2}<d^{t-1}$.  Since $H$ fixes a vertex in each orbit of $N$ on $V\Gamma$ and since $C=C_{\Sym(V\Gamma)}(N)$ permutes the $\mathcal{O}_i$ and the setwise stabiliser in $C$ of each $\mathcal{O}_i$ induces a semiregular action on it, we obtain that $C_{H}(N)=1$. Thus the group $H$ acts faithfully on $N$ by conjugation. Let $\tilde{G}=N\rtimes H$ be the semidirect product
of $N$ by $H$ with respect to this action. We claim that $\tilde{G}$ acts as a group of
automorphisms on $\tilde{\Gamma}$ by setting $$\gamma^{ng}=(\gamma n)^g,$$
for $\gamma\in V\tilde{\Gamma}$, $n\in N$ and $g\in H$. It is straightforward to check 
that this is a well defined action of $\tilde{G}$ on $V\tilde{\Gamma}=N$. Let $n$ be in
$N$, $g$ be in $H$ and $(\gamma,\gamma')$ be an arc of
$\tilde{\Gamma}$, that is, $\gamma\gamma'^{-1}\in S$. By definition of a Cayley
digraph, $(\gamma,\gamma')^{ng}=((\gamma n)^g,(\gamma'
n)^g)=(\gamma^gn^g,\gamma'^{g}n^g)$ is an arc of $\tilde{\Gamma}$ if and only if
$\gamma^gn^g(\gamma '^gn^g)^{-1}=\gamma^g(\gamma'^{g})^{-1}=(\gamma\gamma'^{-1})^g$
lies in $S$. Hence to prove the claim it remains to prove that $H$
leaves the set $S$ invariant under conjugation. Let $s\in S$ (with $\beta_i^s\in \Gamma(\beta_j)$ say) and $g\in H$. We have
$$\beta_i^{s^g}=\beta_i^{g^{-1}sg}=\beta_i^{sg}\in\Gamma(\beta_j)^g=\Gamma(\beta_j^g)=\Gamma(\beta_j),$$
and thus $\beta_i^{s^g}\in \Gamma(\beta_j)$. By definition of $S$, we have $s^g\in S$.
Since $s$ is an arbitrary element of $S$, this shows that $S^g=S$ and the claim is proved.

Let $\Sigma$ be the subgraph of $\Gamma$ induced on the 
set $\{\beta_i^x\mid 1\leq i\leq t, x \in\langle S\rangle\}$. We claim that $\Gamma=\Sigma$. We argue by contradiction. Let $v$ be a vertex in $V\Gamma\setminus V\Sigma$. Since $\Gamma$ is connected, there exists a path $\beta_1=v_1,\ldots,v_u=v$ in $\Gamma$ from $\beta_1$ to $v$. Let $i$ be the minimum such that $v_{i}\notin V\Sigma$. In particular, $i\geq 2$ and $v_{i-1}\in V\Sigma$ and so $v_{i-1}=\beta_k^x$ for some $k\leq t$ and $x\in \langle S\rangle$. Also, the vertex $v_i$ lies in $\mathcal{O}_{k'}$ for some $ k'\leq t$, so as $N$ is transitive on $\mathcal{O}_{k'}$, we have  $v_i=\beta_{k'}^n$ for some $n\in N$. Since $v_{i-1}$ is adjacent to $v_i$, we get $\beta_{k'}^{nx^{-1}}=v_i^{x^{-1}}\in \Gamma(v_{i-1}^{x^{-1}})=\Gamma(\beta_k)$. By definition of $S$, we have $nx^{-1}\in S$. Finally, as $x\in \langle S\rangle$, we obtain $n\in \langle S\rangle$, and $v_i=\beta_{k'}^n\in V\Sigma$, a contradiction. Thus $V\Sigma=V\Gamma$ and hence $\Sigma=\Gamma$.

Since $\Gamma=\Sigma$, the group $\langle S\rangle$ acts transitively on each $N$-orbit. As $H$ fixes a vertex from each $\langle S\rangle$-orbit, we have $C_H(\langle S\rangle)$=1. Let $\tilde{\Gamma}_1$ be the connected component of $\tilde{\Gamma}$ containing $1$
(that is, $V\tilde{\Gamma}_1=\langle S\rangle$) and let $L$ be the permutation group
induced by $H$ on $\tilde{\Gamma}_1$. Since $H$ acts
as a group of automorphisms on $\langle S\rangle=V\tilde{\Gamma}_1$, we obtain $L\cong H/C_H(\langle S\rangle)=H$. Thus $H$ acts faithfully on $V\tilde{\Gamma}_1$ and $H\leq \Aut(\langle S\rangle)$. Also, as $H$ fixes setwise $S$, we get $|H|\leq 
|S|!\leq (df_1(d)^2f_2(d))!$. Therefore $|G_{\alpha}|\leq d^{t-1}|H|\leq d^{f_1(d)-1}(df_1(d)^2f_2(d))!$ for every $\alpha\in V\Gamma$ and the proof is complete.~$\qed$

\begin{remark}\label{rm:1}{\rm Although Theorem~\ref{thm:1} is a very general statement, we often use it when $N$ is vertex-transitive, in which case we can take $f_1(d)=1$ and $f_3(d)=(df_2(d))!$.}
\end{remark}

In the following corollary we single out the special case of Theorem~\ref{thm:1}  most useful for the rest of the paper.

\begin{corollary}\label{cor:good}
Let $(\Gamma,G)$ be in $\mathcal{A}(d)$ and $N$ a normal subgroup of $G$. If $N$ acts regularly on $V\Gamma$, then $|G_\alpha|\leq  d!$ for every $\alpha\in V\Gamma$. 
\end{corollary}
\begin{proof}
If $N$ acts regularly on $V\Gamma$, then $N_\alpha=1$ for every $\alpha\in V\Gamma$ and so $(\Gamma,N)$ is $(1,1)$-bounded. Theorem~\ref{thm:1} with $f_3(d)$ as in Remark~\ref{rmrm} yields $(\Gamma,G)$ is  $(1,d!)$-bounded.
\end{proof}



\section{Auxiliary lemma}\label{sec:auxiliary}

This section contains only Lemma~\ref{lemma:auxiliary}. This very technical result  will be important in the proof of Theorem~\ref{thm:mainqp}.

\begin{lemma}\label{lemma:auxiliary}Let $T$ be a non-abelian simple group, $l\geq 1$, and $R$ a proper subgroup of $T$. Let $m^{(1)}=(m_1^{(1)},\ldots,m_l^{(1)}),\ldots,m^{(d)}=(m_1^{(d)},\ldots,m_l^{(d)})$ be elements of $T^l$ such that, for each $i \in \{1,\ldots,d\}$, the set of entries $\{m_{j}^{(i)}\}_j$ of $m^{(i)}$ contains at most $d$ distinct elements from $T$. Let  $y^{(i)}$ and $z^{(i)}$ be in $R^l$, and set $n^{(i)}:=y^{(i)}m^{(i)}z^{(i)}$, for $i=1,\ldots,d$. If $T^l=\langle n^{(1)},\ldots,n^{(d)}\rangle R^l$, then $l\leq d^{d}|R|^{2d}$.
\end{lemma}
\begin{proof}
Write $y^{(i)}=(y_{1}^{(i)},\ldots,y_l^{(i)})$ and $z^{(i)}=(z_{1}^{(i)},\ldots,z_l^{(i)})$ with $y_{j}^{(i)},z_j^{(i)}\in R$. We denote by $(n^{(i)})_j$  the $j$th coordinate of $n^{(i)}$. Set $U=\langle n^{(1)},\ldots,n^{(d)}\rangle$, and assume that $T^l=UR^l$ and $l>d^d|R|^{2d}$. 

Since  the element $m^{(i)}$ has at most $d$ distinct entries, by the pigeonhole principle we obtain that $m^{(i)}$ has more than $d^{d-1}|R|^{2d}$ coordinates with the same entry. Applying this argument for each $i\in \{1,\ldots,d\}$, we obtain that there exists a set of coordinates $X\subseteq \{1,\ldots,l\}$ with $|X|>|R|^{2d}$ and such that every $m^{(i)}$ is constant on the coordinates from $X$.

 Let $Y$ be the $(d\times |X|)$-array $(y_x^{(i)})_{1\leq i\leq
  d,x\in X}$, and let $Z$ be the  $(d\times |X|)$-array $(z_{x}^{(i)})_{1\leq i\leq
  d,x\in X}$.  The columns of $Y$ and $Z$ are elements of $R^d$.
Therefore, for each of $Y$ and  $Z$, there are at most $|R|^d$ possibilities for each column. As $|X|>|R|^{2d}$, by the pigeonhole principle, there exist distinct $x,x'$ in 
$X$ such that the $x$th and $x'$th columns of $Y$ are
equal and the $x$th and $x'$th columns of $Z$ are
equal. Hence 
$$
y_{x}^{(i)}=y_{x'}^{(i)},\quad z_{x}^{(i)}=z_{x'}^{(i)}
\quad
\textrm{for every }i=1,\ldots,d.$$ 
This yields 
$$
(n^{(i)})_x=y_x^{(i)}m_x^{(i)}z_{x}^{(i)}=y_{x'}^{(i)}m_{x'}^{(i)}z_{x'}^{(i)}=(n^{(i)})_{x'}\quad
\textrm{for every }i=1,\ldots,d.$$ Therefore the projection of $U$ 
to the group $T\times T$ (obtained from taking the $x$th and $x'$th coordinate entries in $T^l$) is contained in
the diagonal subgroup $\{(t,t)\mid t\in T\}$. As $T^l=U R^l$, we obtain
\begin{equation}\label{eq:u}
T\times T=\{(t,t)\mid t\in
T\}(R\times R).
\end{equation}
Let $t$ be an element of $T\setminus R$. From~$(\ref{eq:u})$, we have $(1,t)=(a,a)(b,c)$ for some $a\in T$
and $b,c\in R$. This 
yields  $a=b^{-1}\in R$ and $t=ac\in R$, a
contradiction. This contradiction arose from the assumption
$l>d^d|R|^{2d}$. Hence $l\leq d^d|R|^{2d}$ and the lemma is proved.
\end{proof}

\section{Proof of Theorem~\ref{thm:mainqp}}
\label{redToAS}
In this section we use the subdivision
into eight types (namely HA, HS, HC, SD, CD, TW,
PA and AS) of the finite quasiprimitive permutation groups, and we refer the reader to~\cite{P2} for details. 
Our main tool for dealing with  quasiprimitive groups is Corollary~\ref{cor:good}. Namely, in
Corollaries~\ref{cor:many} and~\ref{cor:SD}, for  $(\Gamma,G)\in \mathcal{A}(d)$,  we give an absolute
upper bound (in terms of the valency $d$) on the
size of the stabiliser of 
a vertex in $G$, if $G$ is quasiprimitive of type HA, HS, HC, TW, SD or CD. We note that these results were proved in~\cite{CLP} in the case where $G$ is locally primitive.

\begin{corollary}\label{cor:many}
Let $(\Gamma,G)$ be in $\mathcal{A}(d)$ with $G$ a quasiprimitive group of type HA, HS,   HC or TW on $V\Gamma$. Then
  $|G_\alpha|\leq d!$ for every $\alpha\in V\Gamma$. 
\end{corollary}

\begin{proof}
As  $G$ is quasiprimitive of type HA, HS, HC or TW, the group $G$
admits a regular normal 
subgroup $N$. From Corollary~\ref{cor:good}, we have
$|G_\alpha|\leq d!$ for every $\alpha\in V\Gamma$.  
\end{proof}

\begin{corollary}\label{cor:SD}
Let $(\Gamma,G)$ be in $\mathcal{A}(d)$ with $G$ a quasiprimitive group of type SD or CD on $V\Gamma$.  Then $|G_\alpha|\leq (dd!)!$ for every $\alpha\in V\Gamma$. 
\end{corollary}

\begin{proof} 
Assume that $G$ is of type SD. Let $N$ be the socle of $G$. So,  $N=T_1\times \cdots \times T_k$ for
some $k\geq 2$, and each $T_i\cong T$ for some non-abelian simple group
$T$.  By definition of type SD, the group $M=T_1\times \cdots \times
T_{k-1}$ acts regularly on $V\Gamma$. Hence $(\Gamma,M)$ is $(1,1)$-bounded. Since $M\trianglelefteq N$, Theorem~\ref{thm:1} with $f_3(d)$ as in Remark~\ref{rmrm} yields that $(\Gamma,N)$ is $(1,f)$-bounded with $f(d)=d!$. As  $N\trianglelefteq G$, a similar application of Theorem~\ref{thm:1}, with $f_3(d)$  as in Remark~\ref{rmrm}, yields that $(\Gamma,G)$ is $(1,f')$-bounded with $f'(d)=(df(d))!=(dd!)!$.

Assume now that $G$ is of type CD. Then the vertex set
$V\Gamma$ admits a cartesian decomposition, that is,
$V\Gamma=\Delta^l$ for some set $\Delta$ and for some $l\geq 2$. Let
$N$ be the socle of $G$. Then 
$N\cong T^{ul}$ for some non-abelian simple group $T$ and some $u\geq
2$. Also, $G$ is permutation isomorphic to a subgroup of the wreath
product $H\wr \Sym(l)$ (endowed with the product action), where $H\subseteq
\Sym(\Delta)$ is quasiprimitive of type SD with socle $T^u$. 

As the socle of $H$ contains a regular normal subgroup isomorphic to
$T^{u-1}$,  the group $N$ 
contains a normal subgroup $M$ isomorphic to 
$T^{(u-1)l}$ acting regularly on $\Delta^l$. Hence $(\Gamma,M)$ is $(1,1)$-bounded. Since $M\trianglelefteq N$, Theorem~\ref{thm:1} with $f_3(d)$ as in Remark~\ref{rmrm} yields that $(\Gamma,N)$ is $(1,f)$-bounded with $f(d)=d!$. As  $N\trianglelefteq G$, a similar application of Theorem~\ref{thm:1} yields that $(\Gamma,G)$ is $(1,f')$-bounded with $f'(d)=(df(d))!=(dd!)!$.
\end{proof}

\begin{remark}\label{rm:1qp}
{\rm It is worth pointing out here that in
Corollaries~\ref{cor:many} and~\ref{cor:SD} there is no
local assumption on $G$. This quite remarkably shows that in
a quasiprimitive group $G$ of type HA, HS, HC, TW, SD or CD acting
vertex-transitively on a connected graph $\Gamma$, the size of the stabiliser of a
vertex is bounded above by a function of the valency of $\Gamma$ (see Theorem~\ref{thm:mainqp}~$(1)$).}
\end{remark}

In the rest of this section we deal with the case that $G$ is of type PA. 
We start with a definition and a lemma required in the proof of Theorem~\ref{prop:PAtype1qp}.

\begin{definition}\label{def:nnq}
{\rm Let $\Gamma$ be a $G$-vertex-transitive graph and $\Sigma$ a system of imprimitivity for $G$ in its action on $V\Gamma$. The \emph{quotient} $\Gamma_\Sigma$ is the graph whose vertices are the blocks $\sigma$ of $\Sigma$, with an edge between two distinct blocks $\sigma$ and $\eta$ of $\Sigma$, if and only if there is an edge of $\Gamma$ between $\alpha$ and $\beta$, for some $\alpha\in\sigma$ and some $\beta\in\eta$. 

The graph $\Gamma_\Sigma$ is $G$-vertex-transitive but (despite the similarity with Definition~\ref{def:nq}) there is no upper bound on the valency of $\Gamma_\Sigma$ in terms of  the valency of $\Gamma$.}
\end{definition}

\begin{lemma}\label{lemma:nrorbits}
Let $(\Gamma,G)$ be in $\mathcal{A}(d)$, $\Sigma$ a system of imprimitivity for the action of $G$ on $V\Gamma$ and $N$ a normal subgroup of $G$ with $N_\sigma$ transitive on $\sigma$ for every $\sigma$ in $\Sigma$. Then the number of orbits of $N_\sigma$ on $\Gamma_{\Sigma}(\sigma)$ is at most $d$, for every vertex $\sigma$ of $\Gamma_\Sigma$. 
\end{lemma}

\begin{proof}
Fix $\sigma$ a vertex of $\Gamma_\Sigma$ and $\alpha$ in $\sigma$. Let $\beta_1,\ldots,\beta_r$ be representatives of the orbits of $G_\alpha$ on $\Gamma(\alpha)$. Then $r\leq |\Gamma(\alpha)|\leq d$. Let $\eta_1,\ldots,\eta_r$ be in $\Sigma$ with $\beta_i$ in $\eta_i$, for $i=1,\ldots,r$. Now we prove two preliminary claims from which the lemma will follow immediately.

\smallskip

\noindent\textsc{Claim~1. }$\Gamma_\Sigma(\sigma)=\bigcup_{i=1}^r\eta_i^{G_\sigma}$.

\noindent By the definitions of $\Gamma_\Sigma$ and of the $\eta_i$, the right hand side is contained in the left hand side. Let $\eta$ be in $\Gamma_\Sigma(\sigma)$. By definition, there exists $\alpha'\in \sigma$ and $\beta\in \eta$ with $\beta\in \Gamma(\alpha')$. Since $G_\sigma$ is transitive on $\sigma$, there exists $g\in G_\sigma$ such that $\alpha=(\alpha')^g$. In particular, $\beta^g\in \Gamma(\alpha)$ and so there exists $h\in G_\alpha$ with $\beta_i=\beta^{gh}$, for some $i\in \{1,\ldots,r\}$. As $\beta_i\in \eta_i$, we get $\eta_i=\eta^{gh}$ with $gh\in G_\sigma$ and so Claim~1 follows.~$_\blacksquare$ 

\smallskip

\noindent\textsc{Claim~2. }The number of orbits of $N_\sigma$ on $\eta_i^{G_\sigma}$ is at most $|G_\alpha:G_{\alpha,\beta_i}|$. 

\noindent Clearly, $|\eta_i^{G_\sigma}|=|G_\sigma:G_{\sigma,\eta_i}|$. Also, as $N_\sigma$ is a normal subgroup of $G_\sigma$, we have that each orbit of $N_\sigma$ on $\eta_i^{G_\sigma}$ has size $|G_{\sigma,\eta_i}N_\sigma:G_{\sigma,\eta_i}|$. Therefore the number of orbits of $N_\sigma$ on $\eta_i^{G_\sigma}$ equals $|G_\sigma:G_{\sigma,\eta_i}N_\sigma|$.

Since by hypothesis  $N_\sigma$ is transitive on $\sigma$, we get  $G_\sigma=G_\alpha N_\sigma$. Note that, all of $G_\sigma$, $N_\sigma$ and $G_{\sigma,\eta_i}N_\sigma$ are transitive on $\sigma$ and hence $|G_\sigma:G_\alpha|=|G_{\sigma,\eta_i}N_\sigma:(G_{\sigma,\eta_i}N_\sigma)\cap G_\alpha|=|\sigma|$. Thus $|G_\sigma:G_{\sigma,\eta_i}N_\sigma|=|G_\alpha:(G_{\sigma,\eta_i}N_\sigma)\cap G_\alpha|\leq |G_\alpha:G_{\alpha,\beta_i}|$, (see Figure~\ref{pic:1}).~$_\blacksquare$

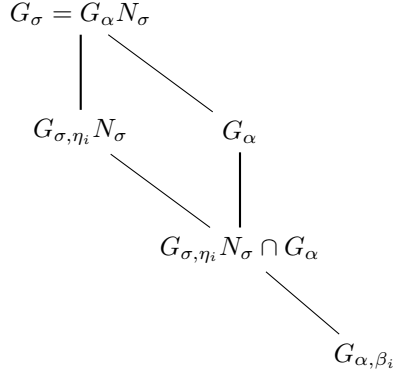
\begin{figure}[!h]
\begin{tikzpicture}[node distance   = 1cm ]
\tikzset{
    myarrow/.style={==, thick}
}  
        \node(A){$G_\sigma=G_\alpha N_\sigma$};
\node[below=of A](B){$G_{\sigma,\eta_i}N_\sigma$};
\node[right=of B](D){$G_\alpha$};
\node[below=of D](E){$G_{\sigma,\eta_i}N_\sigma\cap G_\alpha$};
\node[below=of E](F){};
\node[right=of F](G){$G_{\alpha,\beta_i}$};
\draw(E)--(G);
\draw[myarrow](A)--(B);
\draw(A)--(D);
\draw[myarrow](D)--(E);
\draw(B)--(E);
\end{tikzpicture}
\caption{Subgroup lattice for Lemma~\ref{lemma:nrorbits}}\label{pic:1}
\end{figure}

\smallskip

From Claims~1 and~2, the number of orbits of $N_\sigma$ on $\Gamma_\Sigma(\sigma)$ is at most $\sum_{i=1}^r|G_\alpha:G_{\alpha,\beta_i}|=d$.
\end{proof}

We start our analysis of quasiprimitive groups of type PA by setting some notation (as usual we follow~\cite{P1}
and~\cite{P2}). 

\begin{notation}\label{not:1}
{\rm
Let $\Gamma$ be a connected $G$-vertex-transitive graph of valency $d$ and assume that $G$ is a quasiprimitive group of type PA. The socle 
$N=T_1\times \cdots\times T_l$  of $G$ is isomorphic to $T^l$, where $T$ is a 
non-abelian simple group, and $l\geq 2$. Moreover, there is a
$G$-invariant partition $\Sigma$ of $V\Gamma$ such that the action of
$G$ on $\Sigma$ is faithful and is permutationally isomorphic to the
product action of $G$ on a set $\Delta^l$. By identifying
$\Sigma$ with $\Delta^l$  we have $G\subseteq W=H\wr \Sym(l)$, where
$H\subseteq\Sym(\Delta)$ is an almost simple group with socle $T$ which is
quasiprimitive on $\Delta$, $N$ is the socle of $W$, and $W$ acts on
$\Sigma$ in product action. 
Fix $\delta$ an element of $\Delta$. We denote by
$\sigma$ the element $(\delta,\ldots,\delta)$ of $ 
\Sigma=\Delta^l$. Also, we fix $\alpha_0$ a vertex of $\Gamma$ in
$\sigma$ and let $\beta_1,\ldots,\beta_r$ be representatives of the orbits of $G_{\alpha_0}$ on $\Gamma({\alpha_0})$.  We have  $W_\sigma=H_\delta\wr\Sym(l)$ and $T_\delta$ is a proper subgroup of $T$.  

The group $G$ acts transitively on the set $\{T_1,\ldots,T_l\}$ of
minimal normal subgroups of $N$. As $G=NG_{\alpha_0}$ and $N$ acts
trivially on $\{T_1,\ldots,T_l\}$, we obtain that $G_{\alpha_0}$ acts
transitively on $\{T_1,\ldots,T_l\}$. In the sequel, we use this fact
repeatedly.

It is proved in~\cite{P1} that $N_{\alpha_0}$ is a subdirect subgroup of
$N_\sigma=T_\delta^l$, that is, $N_{\alpha_0}$ projects onto $T_\delta$ for
each of the
$l$ direct factors of $N_\sigma$. Furthermore, the subgroup $G_i:=N_G(T_i)$  has index $l$ in $G$ and $G_i$ induces a subgroup of
$\Sym(\Delta)$; this subgroup is almost simple with socle $T$ and
without loss of generality we may take $H$ to be this subgroup for each $i$. We denote by $\pi_i:G_i\to H$ the projection of $G_i$ onto $H$.}
\end{notation}

\begin{lemma}\label{lemma:proj}$\pi_i(G_i\cap G_{\alpha_0})=\pi_i(G_i\cap G_\sigma)=H_\delta$.
\end{lemma}
\begin{proof}
Since $G_{\alpha_0}\subseteq G_\sigma$, we have $\pi_i(G_i\cap G_{\alpha_0})\subseteq \pi_i(G_i\cap G_\sigma)\subseteq H_\delta$ and hence it suffices to prove that $\pi_i(G_i\cap G_{\alpha_0})=H_\delta$. Set $L=\pi_i(G_i\cap G_{\alpha_0})$.  It was proved in~\cite{P1} that $N_{\alpha_0}$ is a subdirect subgroup of $N_\sigma$ and hence
$\pi_i(N_{\alpha_0})=\pi_i(N_\sigma)=T_\delta\subseteq L$.
Since $G_{\alpha_0}\subseteq W_\sigma=H_\delta\wr\Sym(l)$, we have $L\subseteq H_\delta$. As $N$ is transitive on $V\Gamma$, we have $G=NG_{\alpha_0}$. Hence, from the modular law, we get $G_i=N(G_i\cap G_{\alpha_0})$ and, applying $\pi_i$ on both sides, we have $H=TL$. Thus $H_\delta=H_\delta\cap (TL)=(H_\delta\cap T)L=T_\delta L=L$.
\end{proof}

\begin{theorem}\label{prop:PAtype1qp}
Let $(\Gamma,G)$ be in $\mathcal{A}(d)$ with $G$ a quasiprimitive group of type PA on $V\Gamma$ (as in
Notation~\ref{not:1}). Then $l\leq d^d|T_\delta|^{2d}$. Furthermore, $(\Gamma,G)$ uniquely determines an element $(\Lambda,T)$ in $\mathcal{A}(d)$ with the stabilisers of the vertices of $\Lambda$ conjugate to $T_\delta$. 
\end{theorem}

\begin{proof}
Let $\eta^{i}=(\delta_1^{i},\ldots,\delta_l^{i})$, for $1\leq i\leq s$, be representatives of the orbits of $N_\sigma$ in its action on $\Gamma_\Sigma(\sigma)$. Since $N$ is transitive on $V\Gamma$ and $\Sigma$ is a system of imprimitivity for $G$, we obtain that $N_\nu$ is transitive on $\nu$ for every $\nu\in \Sigma$. So Lemma~\ref{lemma:nrorbits} applies and we have $s\leq d$. Since $N_\sigma=T_\delta\times \cdots\times T_\delta$, it follows from the definition of the $\eta^i$ that

\begin{eqnarray}\label{eq:11}
\Gamma_\Sigma(\sigma)=\bigcup_{i=1}^s(\eta^{i})^{N_\sigma}=\bigcup_{i=1}^s\left((\delta_1^{i})^{T_\delta}\times \cdots \times (\delta_l^{i})^{T_\delta}\right).
\end{eqnarray}

Now we prove two preliminary claims from which the theorem will follow. 

\smallskip

\noindent\textsc{Claim 1. }
$\bigcup_{i=1}^s(\delta_j^{i})^{H_\delta}=\bigcup_{i=1}^{s}(\delta_j^{i})^{T_\delta}$ for every $j\in \{1,\ldots,l\}$.

\noindent As $T_\delta\subseteq H_\delta$, the right hand side is contained in the left hand side. Fix $j$ in $\{1,\ldots,l\}$, $i$ in $\{1,\ldots,s\}$ and $h$ in $H_\delta$. From Lemma~\ref{lemma:proj}, we have $\pi_j(G_j\cap G_\sigma)=H_\delta$. Hence  there exists $c=(h_1,\ldots,h_l)g\in G_j\cap G_\sigma$ with $h_j=h$. Now, $(\eta^i)^c\in \Gamma_\Sigma(\sigma)$ and the $j$th coordinate of $(\eta^i)^c$ is $(\delta_j^i)^h$. So, from~$(\ref{eq:11})$, we get $(\delta_j^i)^h\in \cup_{i=1}^s(\delta_j^i)^{T_\delta}$ and Claim~$1$ follows.~$_\blacksquare$

\smallskip

\noindent\textsc{Claim 2. }
$\bigcup_{i=1}^s(\delta_j^{i})^{T_\delta}=\bigcup_{i=1}^{s}(\delta_1^{i})^{T_\delta}$ for every $j\in \{1,\ldots,l\}$.

\noindent Fix $j$ in $\{1,\ldots,l\}$. Since the left and the right hand side are $T_\delta$-invariant, it suffices to show that $\delta_1^v\in \cup_{i=1}(\delta_j^i)^{T_\delta}$ for every $v\in \{1,\ldots,s\}$. (This will imply that the right hand side is contained in the left hand side, and an analogous argument proves the reverse inclusion.) Fix $v$ in $\{1,\ldots,s\}$. Recall that $G_\sigma$ is transitive on $\{T_1,\ldots,T_l\}$. 
So, there exists $$c=(h_1,\ldots,h_l)g\in G_\sigma\subseteq H_\delta\wr \Sym(l)$$
such that $g$ conjugates $T_1$ to $T_j$, that is, $1g=j$. Then $\Gamma_\Sigma(\sigma)$ contains $\eta^v$ and hence contains

$$(\eta^v)^c=((\delta_{1g^{-1}}^{v})^{h_{1g^{-1}}},\ldots, (\delta_{lg^{-1}}^{v})^{h_{lg^{-1}}}),$$
which has $j$th entry
$(\delta_{jg^{-1}}^{v})^{h_{jg^{-1}}}=(\delta_1^{v})^{h_1}$. From~$(\ref{eq:11})$, we obtain that the $j$th entries of the elements of $\Gamma_\Sigma(\sigma)$ lie in  $\cup_{i=1}^s(\delta_j^i)^{T_\delta}$. Therefore $$(\delta_1^v)^{h_1}\in \bigcup_{i=1}^s(\delta_j^{i})^{T_\delta}.$$
Now, as $h_1^{-1}\in H_\delta$, Claim~1 yields $\delta_1^v\in \cup_{i=1}^s(\delta_j^i)^{T_\delta}$ and 
Claim~2 follows.~$_\blacksquare$

\smallskip

Claim~2 yields that the $ls$ elements $\{\delta_j^i\}_{j,i}$ lie in at most $s$ distinct $T_\delta$-orbits. In particular, since $N_\sigma=T_\delta^l$, for each $i\in \{1,\ldots,s\}$, the $N_\sigma$-orbit $(\eta^i)^{N_\sigma}$ contains an element with at most $s$ distinct entries from $\Delta$. Therefore, by replacing $\eta^i$ with a suitable element from $(\eta^i)^{N_\sigma}$ if necessary, we may assume that there are at most $s$ distinct elements of $\Delta$ among the $l$ entries of $\eta^i$.

Since $N$ is transitive on $V\Gamma$ we may choose  $m_i\in N$ such that $\sigma^{m_i}=\eta^{i}$. Moreover, since $\sigma=(\delta,\dots,\delta)$, and since $\eta^{i}$ has at most $s$ distinct entries from $\Delta$, the element $m_i$ can be chosen so that its $l$ coordinates contain at most $s$ distinct entries from $T$.

For each $\beta\in \Gamma({\alpha_0})$, let $n_\beta$ be an element of $N$ with $\beta={\alpha_0}^{n_\beta}$. Set $U=\langle n_\beta\mid \beta\in \Gamma({\alpha_0})\rangle$. Let $\overline{\Gamma}$ be the subgraph of $\Gamma$ induced on the set ${\alpha_0}^U$. We claim that $\Gamma=\overline{\Gamma}$. By the definitions of $\overline{\Gamma}$ and $U$, we have  ${\alpha_0}\in V\overline{\Gamma}$ and $\Gamma({\alpha_0})\subseteq V\overline{\Gamma}$. Therefore, since $\overline{\Gamma}$ is $U$-vertex-transitive, every vertex of $\overline{\Gamma}$ has valency $|\overline{\Gamma}({\alpha_0})|=|\Gamma({\alpha_0})|$. Since $\Gamma$ is connected, this yields $\Gamma=\overline{\Gamma}$. In particular, $U$ acts transitively on $V\Gamma$ and so $N=U N_{\alpha_0}$. As $N_{\alpha_0}$ is a subgroup of $N_\sigma$, we have $N=UN_\sigma$.

Fix $\beta$ in $\Gamma({\alpha_0})$. Since ${\alpha_0}^{n_\beta}\in \Gamma({\alpha_0})$, we get $\sigma^{n_\beta}\in \Gamma_{\Sigma}(\sigma)$ and hence $\sigma^{n_\beta}\in (\eta^{i_\beta})^{N_\sigma}$ for some $i_\beta\in \{1,\ldots, s\}$. In particular, there exists $z_\beta\in N_\sigma$ such that 
$$
\sigma^{n_{\beta}z_\beta}=\eta^{i_\beta}.
$$
Since by the definition of $m_{i_\beta}$ we have $\sigma^{m_{i_\beta}}=\eta^{i_\beta}$, we obtain $\sigma=\sigma^{n_\beta z_\beta m_{i_\beta}^{-1}}$, that is to say, $y_\beta:=n_\beta z_\beta m_{i_\beta}^{-1}\in N_\sigma$. Therefore, for every $\beta\in\Gamma({\alpha_0})$, there exists $i_\beta\in \{1,\ldots,s\}$ and $y_\beta,z_\beta\in N_\sigma$ such that 

$$
n_\beta=y_\beta m_{i_\beta}(z_\beta)^{-1}.$$
Now Lemma~\ref{lemma:auxiliary} applied to $m_1,\ldots,m_s\in N=T^l$ gives $l\leq d^d|T_\delta|^{2d}$.

Set $M_i=T_1\times \cdots \times T_{i-1}\times T_{i+1}\times \cdots \times T_l$ for $i\in\{1,\ldots,l\}$. Since $M_i$ is a maximal normal subgroup of $N$, we obtain that either $M_i$ is transitive on $V\Gamma$ or $M_i$ is $1$-closed in $N$. As $N_\alpha$ is a subdirect subgroup of $N_\sigma$, we have $N_\alpha M_i=N_\sigma M_i=T_\delta\times M_i<N$ and so $M_i$ is $1$-closed in $N$. Therefore, by Definition~\ref{def:nq}, the pair $(\Gamma_{M_i},N/M_i)$ lies in $\mathcal{A}(d)$ for each $i\in \{1,\ldots,l\}$. Furthermore, since $G$ acts transitively on $\{M_1,\ldots,M_l\}$, we obtain that $\Gamma_{M_i}\cong \Gamma_{M_j}$ for every $i,j\in \{1,\ldots,l\}$. This shows that $(\Gamma,G)$ uniquely determines the element $(\Gamma_{M_1},N/M_1)$ of $\mathcal{A}(d)$
up to isomorphism. Finally, the stabiliser in  $N/M_1$ of the vertex $\alpha^{M_1}$ of $\Gamma_{M_1}$ is $N_\alpha M_1/M_1=(T_\delta\times M_1)/M_1\cong T_\delta$ and theorem is proved.
\end{proof}

Now we are ready to prove Theorem~\ref{thm:mainqp}.

\smallskip

\noindent\emph{Proof of Theorem~\ref{thm:mainqp}. }
If $G$ is of type HA, HS, HC, TW, SD or CD, then by Corollaries~\ref{cor:many} and~\ref{cor:SD} the graph $(\Gamma,G)$ is $(dd!)!$-bounded. 

Assume that $G$ is of type AS, and set $\Lambda=\Gamma$ and $T=\soc(G)$. Clearly, $(\Lambda,T)$ is uniquely determined by $(\Gamma,G)$. Also, as $T$ is transitive on $V\Gamma$, from Theorem~\ref{thm:1} with $f_3(d)$ as in Remark~\ref{rmrm}, if $(\Lambda,T)$ is $g$-bounded, then $(\Gamma,G)$ is $f$-bounded where $f(d)=(dg(d))!$. In particular, as $f(d)\leq \widehat{g}(d)$, we get that $(\Gamma,G)$ is $\widehat{g}$-bounded. Conversely, if $(\Gamma,G)$ is $f$-bounded, then $(\Lambda,T)$ is also $f$-bounded.

Finally assume that $G$ is of type PA and let $(\Lambda,T)$ be as in Theorem~\ref{prop:PAtype1qp}. We use Notation~\ref{not:1}. From Theorem~\ref{prop:PAtype1qp}, we have $l\leq d^d|T_\lambda|^{2d}$ with $\lambda\in V\Lambda$. Assume that $(\Gamma,G)$ is $f$-bounded for some $f:\mathbb{N}\to \mathbb{N}$. As $N_\alpha$ is a subdirect subgroup of $N_\sigma\cong T_\delta^l$ and $|T_\delta|=|T_\lambda|$, we have that $|T_\lambda|\leq |N_\alpha|\leq f(d)$ and so $(\Lambda,T)$ is $f$-bounded. Finally, assume that $(\Lambda,T)$ is $g$-bounded for some $g:\mathbb{N}\to \mathbb{N}$. Now, $|N_\alpha|\leq |N_\sigma|=|T_\lambda|^l\leq g(d)^{l}\leq g(d)^{d^dg(d)^{2d}}$ and so $(\Gamma,N)$ is $f'$-bounded with $f'(d)=g(d)^{d^dg(d)^{2d}}$. Finally, Theorem~\ref{thm:1} with $f_3(d)$ as in Remark~\ref{rmrm} yields that $(\Gamma,G)$ is $f$-bounded for $f(d)=(dg(d)^{d^dg(d)^{2d}})!=\widehat{g}(d)$.~$\qed$

\section{Proof of Theorem~\ref{thm:mainbiqp}}\label{redToAS2}

Recall that a permutation group $G$ is biquasiprimitive if every non-trivial normal subgroup of $G$ has at most two orbits and $G$ does have a normal subgroup with exactly two orbits. In~\cite[Theorem~$1.1$]{P1b}, the structure of a biquasiprimitive group is described in detail and here we use~\cite{P1b} as a reference. We set some notation for the rest of the paper, which follows~\cite{P1b}.

\begin{notation}\label{not:11}{\rm 
 For a finite biquasiprimitive permutation group $G$ on $\Omega$, there is at least one non-trivial intransitive normal subgroup $N$ (since $G$ is not quasiprimitive) and $N$ must therefore have two orbits, say $\Delta,\Delta'$. Each element of $G$ either fixes set-wise these two orbits or interchanges them. Thus the elements of $G$ that fix $\Delta,\Delta'$ setwise form a subgroup $G^+$ of index $2$, and $G^+$ induces a transitive permutation group $H$ on $\Delta$. By the embedding theorem for permutation groups, $G$ is conjugate in $\Sym(\Omega)$ to a subgroup of the wreath product $H\wr \Sym(2)=(H\times H)\rtimes \Sym(2)$. The set $\Omega$ may be identified with $\Delta\times \{1,2\}$ such that, for  $(y_1,y_2)$ in the base group $H\times H$, and $(1,2)\in \Sym(2)$,
$$(\delta,i)^{(y_1,y_2)}=(\delta^{y_i},i)\quad\textrm{and}\quad (\delta,i)^{(1,2)}=(\delta,i^{(1,2)})$$ 
for all $(\delta,i)\in \Omega$. Theorem~$1.1$ in~\cite{P1b} (which we report below) defines various distinct possibilities for $\soc(G)$. 

Let $M$ be a group. For each $\varphi\in \Aut(M)$, we denote by $\Diag_\varphi(M\times M)$ the full diagonal subgroup $\{(x,x^\varphi)\mid x\in M\}$ of $M\times M$. We write $\iota_x:y\mapsto x^{-1}yx$ for the inner automorphism induced by the element $x\in M$.}
\end{notation}

Before stating Theorem~\ref{thm:cheryl} we remark that if $\Gamma$ is a $G$-vertex-biquasiprimitive graph, then $\Gamma$ is not necessarily bipartite with bipartition $\{\Delta\times \{1\},\Delta\times\{2\}\}$. Indeed, the hypothesis of $\Gamma$ being vertex-transitive does not imply that every edge of $\Gamma$ joins vertices from distinct $G^+$-orbits (but this is the case if $\Gamma$ is connected and $G$-arc-transitive). 

\begin{theorem}[{\cite[Theorem~$1.1$ and Theorem~$1.2$]{P1b}}]
\label{thm:cheryl}
Let $G$ be a  biquasiprimitive group on $\Omega$, and  
$H$  the permutation group induced by $G^+$ on $\Delta\times \{1\}$ (as in Notation~\ref{not:11}). Assume that $G_{(\alpha,1)}$ is intransitive on $\Delta\times\{2\}$ for $\alpha\in \Delta$. Replacing $G$ by a conjugate in $\Sym(\Omega)$ if necessary, $G\leq H\wr \Sym(2)$, $G\setminus G^+$ contains an element $g=(x,1)(1,2)$ for some $x\in H$, $G^+=\Diag_\varphi(H\times H)$ where $\varphi\in\Aut(H)$ and $\varphi^2=\iota_x$, and one of the following holds.
\begin{description}
\item[$(A)$]$H$ is quasiprimitive on $\Delta$. 
\item[$(B)$]$H$ is not quasiprimitive on $\Delta$; there exists an intransitive minimal normal subgroup $R$ of $H$ such that $R^\varphi\neq R$, $M=R\times R^\varphi$ is a transitive normal subgroup of $H$, and $N=\Diag_\varphi(M\times M)$ is a minimal normal subgroup of $G$; and one of:
\begin{itemize}
\item[$(i)$]$\soc(G)=N$.
\item[$(ii)$]$\soc(G)=N\times\overline{N}$, where $N,\overline{N}$ are isomorphic non-abelian minimal normal subgroups of $G$, and $\overline{N}=\Diag_\varphi(\overline{M}\times\overline{M})$; $M$, $\overline{M}$ are isomorphic regular normal subgroups of $H$, $\soc (H)=M\times \overline{M}$, and $\overline{M}=\overline{R}\times\overline{R}^\varphi$ for an intransitive minimal normal subgroup $\overline{R}$ of $H$.
\end{itemize}
\end{description}
\end{theorem}
\begin{remark}\label{rm:clearer}{\rm 
We warn the reader that~\cite[Theorem~$1.1$]{P1b} describes the structure of any finite biquasiprimitive permutation group (that is, $G_{(\alpha,1)}$ is not necessarily assumed to be intransitive on $\Delta\times\{2\}$ for $\alpha\in \Delta$). The refinement of~\cite[Theorem~$1.1$]{P1b} in~\cite[Theorem~$1.2$]{P1b} is concerned with biquasiprimitive groups acting transitively on the vertices of a bipartite graph $\Gamma$. The statement of~\cite[Theorem~$1.2$]{P1b} requires $G$ to be arc-transitive on $\Gamma$, but the proof uses simply that $G_{(\alpha,1)}$ is intransitive on $\Delta\times \{2\}$ for $\alpha\in \Delta$. Therefore a proof of Theorem~\ref{thm:cheryl} combines~\cite[Theorem~$1.1$ and Theorem~$1.2$]{P1b} together with this remark. 
}
\end{remark}

In the rest of this section we use the subdivision into types given in Theorem~\ref{thm:cheryl} and  we prove Theorem~\ref{thm:mainbiqp}.  We start our analysis with the easiest example.

\begin{lemma}\label{lemma:silly}
Let $(\Gamma,G)\in\mathcal{A}(d)$ with $G$ biquasiprimitive on $V\Gamma$ and $G_{(\alpha,1)}$ transitive on $\Delta\times\{2\}$ for $\alpha\in \Delta$. Then $|G_{(\alpha,i)}|\leq d!(d-1)!$ for every $\alpha\in \Delta$ and $i=1,2$.
\end{lemma}
\begin{proof}
Fix $\alpha$ in $\Delta$. As $\Gamma$ is connected, $(\alpha,1)$ has a neighbour, $(\alpha',2)$ say, in $\Delta\times\{2\}$. Since $G_{(\alpha,1)}$ is transitive on $\Delta\times \{2\}$ and $G\leq \Aut(\Gamma)$, we obtain $\Gamma((\alpha,1))\supseteq \Delta\times \{2\}$, $d\geq |\Delta|$ and $|G_{(\alpha,1)}|\leq d!(d-1)!$.
\end{proof}

From now on, we may assume that $G_{(\alpha,1)}$ is intransitive on $\Delta\times\{2\}$ and so Theorem~\ref{thm:cheryl} applies.

\begin{notation}\label{not:distance2}{\rm
Fix $(\Gamma,G)$ in $\mathcal{A}(d)$ with $G$  biquasiprimitive on $V\Gamma$ and $G_{(\delta,1)}$ intransitive on $\Delta\times \{2\}$ for $\delta\in\Delta$ (as in Notation~\ref{not:11}). We denote by $\Gamma^{\Delta,1}$ (respectively, $\Gamma^{\Delta,2}$) the graph whose vertices are the elements of  $\Delta$, with an edge between two distinct elements $\delta_1$ and $\delta_2$  of $\Delta$ if and only if the distance of $(\delta_1,1)$ from $(\delta_2,1)$  (respectively, $(\delta_1,2)$ from $(\delta_2,2)$) in $\Gamma$ is at most $2$.

Since $\Gamma$ is connected of valency $d$, the graph $\Gamma^{\Delta,i}$ is connected of valency at most $d(d-1)$ for each $i\in\{1,2\}$. The subgroup $G^+=\Diag_\varphi(H\times H)$ of $G$ that fixes setwise $\Delta\times\{1\}$ and $\Delta\times\{2\}$ acts transitively on $\Delta\times\{i\}$ and the action of $G^+$ on $\Delta\times\{i\}$ is equivalent to the action of $H$ on $\Delta$ and preserves the edge set of $\Gamma^{\Delta,i}$ for each $i\in\{1,2\}$. Therefore $\Gamma^{\Delta,i}$ is an $H$-vertex-transitive graph of valency at most $d(d-1)$, that is, $(\Gamma^{\Delta,i},H)\in \mathcal{A}(d(d-1))$ for each $i\in \{1,2\}$. As $g=(x,1)(1,2)\in G$ and $(\Delta\times\{1\})^g=\Delta\times\{2\}$, we have $\Gamma^{\Delta,1}\cong\Gamma^{\Delta,2}$. Hence, without risk of ambiguity, we write simply $\Gamma^\Delta$ for $\Gamma^{\Delta,1}$.

The subgroup $G^+$ is the unique subgroup of $G$ of index $2$ except for the case where $|V\Gamma|=4$ and $G=Z_2\times Z_2$ (see~\cite[Remarks~$1.1$~$(1)$]{P1b}), in which case Theorem~\ref{thm:mainbiqp} is obvious. Therefore, in the remainder of the section, we assume that $G^+$ is the unique subgroup of index $2$ in $G$ and hence $\Delta$ is uniquely determined by $G$.

Summing up, the element $(\Gamma,G)$ of $\mathcal{A}(d)$ uniquely determines the element $(\Gamma^\Delta,H)$ of $\mathcal{A}(d(d-1))$ with $G_{(\alpha,i)}\cong H_\alpha$ for every $\alpha\in \Delta$ and for every $i\in\{1,2\}$.}
\end{notation}

\begin{lemma}\label{tedious}
Let $(\Gamma,G)\in \mathcal{A}(d)$ and $(\Gamma^\Delta,H)\in \mathcal{A}(d(d-1))$ as in Notation~\ref{not:distance2}. If $(\Gamma^\Delta,H)$ is $f$-bounded, then $(\Gamma,G)$ is $f'$-bounded where $f'(d)=f(d(d-1))$, and if $(\Gamma,G)$ is $f$-bounded, then $(\Gamma^\Delta,H)$ is $\tilde{f}$-bounded where $\tilde{f}$ is as in Definition~\ref{functions}. 
\end{lemma}

\begin{proof}We have $|G_{(\alpha,i)}|=|H_\alpha|$ for every $\alpha\in\Delta$ and $i\in \{1,2\}$. As $(\Gamma,G)\in \mathcal{A}(d)$ and $(\Gamma^\Delta,H)\in \mathcal{A}(d(d-1))$, from Definition~\ref{def:Ad}, we have to show that $|G_{(\alpha,i)}|$ is bounded above by a function of $d$ if and only if $|H_\alpha|$ is bounded above by a function of $d(d-1)$. If $(\Gamma^\Delta,H)$ is $f$-bounded, then $|H_\alpha|\leq f(d')$ where $d'=d(d-1)$ and so $|G_{(\alpha,i)}|\leq f'(d)$ where $f'(d)=f(d(d-1))$.
Conversely, if $(\Gamma,G)$ is $f$-bounded, then $|G_{(\alpha,i)}|\leq f(d)$ and so $|H_\alpha|\leq f(d)=\tilde{f}(d(d-1))$.
\end{proof}

\begin{theorem}\label{thm:Bii}
Let $(\Gamma,G)$ be in $\mathcal{A}(d)$ with $G$ biquasiprimitive of type~$(A)$ or~$(B)$~$(ii)$ on $V\Gamma$ (notation as in Theorem~\ref{thm:cheryl}). Then Theorem~\ref{thm:mainbiqp} holds for $(\Gamma,G)$.
\end{theorem}

\begin{proof}
Assume that  $G$ is of type~$(A)$, that is, $H$ is quasiprimitive on $\Delta$. In particular, Theorem~\ref{thm:mainqp} applies to $(\Gamma^\Delta,H)$. If Part~$(1)$ of Theorem~\ref{thm:mainqp} holds for $(\Gamma^\Delta,H)$, then from Notation~\ref{not:distance2} and Lemma~\ref{tedious} we obtain that $(\Gamma,G)$ is $f$-bounded with $f(d)=(d(d-1)(d(d-1))!)!$ and Part~$(1)$ of Theorem~\ref{thm:mainbiqp} holds for $(\Gamma,G)$. Assume that Part~$(2)$ of Theorem~\ref{thm:mainqp} holds for $(\Gamma^\Delta,H)$ and let $(\Lambda,T)$ be as in Theorem~\ref{thm:mainqp}~$(2)$. We show that Part~$(2)$ of Theorem~\ref{thm:mainbiqp} holds by taking $\Lambda_1=\Lambda_2=\Lambda$. If $(\Lambda_i,T)$ is $g_i$-bounded, then by Theorem~\ref{thm:mainqp}, for each $i\in\{1,2\}$,
we have that $(\Gamma^\Delta,H)$ is $\widehat{g_i}$-bounded and hence from Notation~\ref{not:distance2} and Lemma~\ref{tedious}, $(\Gamma,G)$ is $f$-bounded with $f(d)=\widehat{g_i}(d_0)$ and $d_0=d(d-1)$. In particular, $(\Gamma,G)$ is $\overline{g_1\ast g_2}$-bounded. Conversely, assume that $(\Gamma,G)$ is 
 $f$-bounded. By Notation~\ref{not:distance2}, $|H_\alpha|=|G_{(\alpha,1)}|\leq f(d)$ and $(\Gamma^\Delta,H)\in \mathcal{A}(d(d-1))$. So $(\Gamma^\Delta,H)$ is $\tilde{f}$-bounded. Now Theorem~\ref{thm:mainqp}~$(2)$ yields that $(\Lambda,T)$ is $\tilde{f}$-bounded. 

Assume now that $G$ is of type~$(B)$~$(ii)$. From Theorem~\ref{thm:cheryl}, $M$ is a regular normal subgroup of $H$. Therefore from Corollary~\ref{cor:good} applied to $M$, $H$ and $\Gamma^\Delta$, we have $|H_\alpha|\leq (d(d-1))!$ for every $\alpha\in \Delta$. Hence $(\Gamma,f)$ is $f$-bounded with $f(d)=(d(d-1))!$ and Theorem~\ref{thm:mainbiqp}~$(1)$ holds for $(\Gamma,G)$.
\end{proof}

\begin{remark}\label{rm:2qp}
{\rm It is worth pointing out here that Theorem~\ref{thm:Bii} together with Remark~\ref{rm:1qp} show that there is a function $f:\mathbb{N}\to \mathbb{N}$ such that, if $G$ is of type $(A)$ (with $H$ a quasiprimitive group of type HA, HS, HC, TW, SD or CD) or if $G$ is of type~$(B)$~$(ii)$, then $(\Gamma,G)$ is $f$-bounded. Furthermore a suitable function $f$ can be explicitly determined from Corollaries~\ref{cor:many} and~\ref{cor:SD} and Theorem~\ref{thm:Bii}.} 
\end{remark}

Throughout the remainder this section we assume Notations~\ref{not:11} and~\ref{not:distance2} and we fix $(\Gamma,G)$ in $\mathcal{A}(d)$ with $G$ biquasiprimitive of type~$(B)~(i)$ on $V\Gamma$ (notation as in Theorem~\ref{thm:cheryl}). Write $S=R^\varphi$ and recall that $R$ and $S$ are intransitive minimal normal subgroups of $H$. Furthermore $\soc(H)=M=R\times S$ is transitive on $\Delta$.

\begin{lemma}\label{lemma:Rabelian}
If $R$ is abelian, then $(\Gamma,G)$ is $f$-bounded with $f(d)=(d(d-1))!$.
\end{lemma}
\begin{proof}
If $R$ is abelian, then  so is $S$ and  $M$. In particular, $M$ is an abelian normal transitive subgroup of $H$. Therefore from Corollary~\ref{cor:good} applied to $M$, $H$ and $\Gamma^\Delta$, we have $|H_\alpha|\leq (d(d-1))!$. Hence $(\Gamma,G)$ is  $f$-bounded with $f(d)=(d(d-1))!$.
\end{proof}

\begin{notation}\label{not:3}
{\rm
Given Lemma~\ref{lemma:Rabelian}, from now on we may assume that $R$ is the direct product of $l$ isomorphic non-abelian simple groups, each isomorphic to $T$ say. Write $R=T_{R,1}\times\cdots\times T_{R,l}$ and $S=T_{S,1}\times \cdots\times T_{S,l}$ with $T_{R,j}\cong T_{S,j}\cong T$ for every $j\in \{1,\ldots,l\}$. Since $R$ and $S$ are minimal normal subgroups of $H$, the group $H$ permutes transitively the sets $\{T_{R,j}\}_j$ and $\{T_{S,j}\}_j$. Let $\alpha \in \Delta$ and denote by $\pi_{R,j,\alpha}:M_\alpha\to T_{R,j}$ the natural projection on the $j$th coordinate of $R$. Similarly, denote by $\pi_{S,j,\alpha}:M_\alpha\to T_{S,j}$ the natural projection on the $j$th coordinate of $S$.}
\end{notation}

\begin{lemma}\label{lemma:projRSsurj}
If $\pi_{R,j,\alpha}$ or $\pi_{S,j,\alpha}$ is surjective for some $j\in \{1,\ldots,l\}$ and for some $\alpha\in \Delta$, then $(\Gamma,G)$ is $f$-bounded with $f(d)=(d^2((d(d-1))!)^2)!$.
\end{lemma}

\begin{proof}
We prove three claims from which the lemma will follow.

\smallskip

\noindent\textsc{Claim~1. }Let $L$ be normal in $H$. If $L_{\alpha}\neq 1$, then $L_\alpha^{\Gamma^\Delta(\alpha)}\neq 1$. 

\noindent Since $H$ is transitive on $\Delta$, $L$ is normal in $H$ and $L_{\alpha}\neq 1$, we have $L_{\alpha'}\neq 1$ for every $\alpha'\in \Delta$. Fix $\alpha$ in $\Delta$. As $\Gamma^\Delta$ is connected, there exists a vertex $\alpha'$ of $\Gamma^\Delta$ such that $L_\alpha$ fixes $\alpha'$ and $L_\alpha^{\Gamma^\Delta(\alpha')}\neq 1$. Replacing $\alpha'$ if necessary, we may assume that $\alpha'$ is chosen so that its distance $r$ from $\alpha$ is minimal. As $L_{\alpha}\leq L_{\alpha'}$ and $\alpha$, $\alpha'$ are conjugate under $H$, we obtain $L_{\alpha}=L_{\alpha'}$. By minimality of $r$, this yields $r=0$, $\alpha'=\alpha$ and 
$L_{\alpha}^{\Gamma^\Delta(\alpha)}=L_{\alpha}^{\Gamma^\Delta(\alpha')}\neq 1$, and the claim is proved.~$_\blacksquare$

\smallskip
 
\noindent\textsc{Claim~$2$. }If $M_\alpha=R_\alpha\times S_\alpha$, $R_\alpha\neq 1$ and $R_\alpha$ 
is a minimal normal subgroup of $H_\alpha$, then $(\Gamma,G)$ is $f$-bounded with $f(d)=(d^2((d(d-1))!)^2)!$.

\noindent As $R_\alpha\neq 1$, from Claim~$1$ we have that $R_\alpha$ acts non-trivially on 
$\Gamma^\Delta(\alpha)$. As $R_{\alpha}$ is a minimal normal subgroup of 
$H_\alpha$, we obtain that $R_{\alpha}$ acts faithfully on $\Gamma^\Delta(\alpha)$ and 
hence $|R_{\alpha}|\leq |\Gamma^\Delta(\alpha)|!\leq (d(d-1))!$ for every $\alpha\in \Delta$. 

Write $\widehat{R}=\{(r,r^\varphi)\mid r\in R\}\subseteq G^+$ and $\widehat{S}=\{(s,s^\varphi)\mid s\in S\}\subseteq G^+$. Since $R^\varphi=S$, $\varphi^2=\iota_x$ and $g=(x,1)(1,2)$, we have 

\begin{eqnarray*}
(r,r^\varphi)^g&=&(r,r^\varphi)^{(x,1)(1,2)}=(r^x,r^\varphi)^{(1,2)}=(r^\varphi,r^{x})
=((r^{\varphi}),(r^\varphi)^{\varphi})
\end{eqnarray*}
and $\widehat{R}^g=\widehat{S}$. Therefore $(\widehat{R}_{(\alpha,1)})^g=\widehat{S}_{(\alpha,1)^g}=\widehat{S}_{(\alpha^x,2)}$ and so $|\widehat{S}_{(\alpha^x,2)}|=|\widehat{R}_{(\alpha,1)}|=|R_\alpha|\leq (d(d-1))!$ for every $\alpha\in \Delta$. Let $(\beta,2)$ be in $\Gamma((\alpha,1))$. As $\widehat{S}_{(\alpha,1),(\beta,2)}\subseteq \widehat{S}_{(\beta,2)}$ and  $|\widehat{S}_{(\alpha,1)}:\widehat{S}_{(\alpha,1),(\beta,2)}|\leq d$, we obtain $|S_\alpha|=|\widehat{S}_{(\alpha,1)}|\leq d(d(d-1))!$. Therefore $|M_{\alpha}|=|R_{\alpha}\times S_{\alpha}|\leq d((d(d-1))!)^2$. From Theorem~\ref{thm:1} (with $f_3(d)$ as in Remark~\ref{rmrm})  applied to $H$, $M$ and $\Gamma^\Delta$, we get $|H_\alpha|\leq (d^2((d(d-1))!)^2)!$. Hence $(\Gamma,G)$ is $f$-bounded with $f(d)=(d^2((d(d-1))!)^2)!$.~$_\blacksquare$

\smallskip

\noindent\textsc{Claim~$3$. }If $\pi_{R,j,\alpha}$ is surjective for some $j\in \{1,\ldots,l\}$ and for some $\alpha\in \Delta$, then $\pi_{R,j,\alpha}$ is surjective for every $j\in \{1,\ldots,l\}$ and for every $\alpha\in \Delta$. A similar claim holds replacing $R$ by $S$.

\noindent Assume that $\pi_{R,j,\alpha}$ is surjective for some $j$ and for some $\alpha$. Since $M$ is transitive on $\Delta$ and $M$ acts trivially on the set $\{T_{R,j}\}_j$, the projection $\pi_{R,j,\beta}$ is surjective for every $\beta\in \Delta$. Since  $H=H_\alpha M$, the group $H_\alpha$ acts transitively on $\{T_{R,j}\}_j$. As $M_\alpha$ is normal in $H_\alpha$, we obtain that $\pi_{R,j',\alpha}$  is surjective for every $j'\in \{1,\ldots,l\}$.~$_\blacksquare$

\smallskip

Now we continue the proof of the lemma. Replacing $R$ by $S$ if necessary, we may assume that $\pi_{R,j',\alpha'}$ is surjective for some $j'$ and for some $\alpha'$. From Claim~$3$, $\pi_{R,j,\alpha}$ is surjective for every $j$ and for every $\alpha$. We divide the proof  in various cases, depending on whether the projection $\pi_{S,j',\alpha'}$ is also surjective. 

Assume that $\pi_{S,j',\alpha'}$ is surjective for some $j'$ and for some $\alpha'$. Then, from Claim~$3$, the mapping $\pi_{S,j,\alpha}$ is surjective for every $j$ and for every $\alpha$. This yields that $M_\alpha$ is a subdirect subgroup of $M\cong T^{2l}$, and hence by Scott's Lemma (see~\cite[Lemma, page~$328$]{Scott}), $M_\alpha$ is isomorphic to a direct product of $s\geq 1$ copies of $T$. Namely, $M_\alpha=D_{1}\times \cdots \times D_s$ where $D_i\cong T$ for each $i$. (Specifically, each $D_i$ is a diagonal subgroup of some subproduct of $T^{2l}\cong M$.)

Suppose first that $H_\alpha$ acts transitively on the set $\{D_i\}_i$ by conjugation. Then $M_\alpha$ is a minimal normal subgroup of $H_\alpha$. This implies that either $M_\alpha$ acts faithfully or trivially on $\Gamma^\Delta(\alpha)$. Since $M_\alpha\neq 1$, Claim~$1$ gives that $M_{\alpha}$ acts faithfully on $\Gamma^\Delta(\alpha)$ and so $|M_{\alpha}|\leq (d(d-1))!$. From Theorem~\ref{thm:1} with $f_3(d)$ as in Remark~\ref{rmrm} applied to $H$, $M$ and $\Gamma^\Delta$, we get $|H_\alpha|\leq (d((d(d-1))!))!$. Hence $(\Gamma,G)$ is $f$-bounded with $f(d)=(d((d(d-1))!))!$. 

Assume  now that  $H_\alpha$ is intransitive on $\{D_i\}_i$.  As $H_\alpha$ permutes transitively the sets $\{T_{R,j}\}_j$ and $\{T_{S,j}\}_j$, the group $H_\alpha$ has two orbits on  $\{D_i\}_i$ and each $D_i$ is contained in $R$ or in $S$. Relabelling the indices of the subgroups $D_i$ if necessary, we may assume that $\{D_1,\ldots,D_k\}$ and $\{D_{k+1},\ldots,D_s\}$ are the two orbits of $H_\alpha$ on $\{D_i\}_i$, with $D_i\subseteq R\cap M_\alpha$ for $i\in \{1,\ldots,k\}$ and  $D_i\subseteq S\cap M_\alpha$ for $i\in \{k+1,\ldots,s\}$. Therefore 
\begin{eqnarray*}
R_\alpha\times S_\alpha&\subseteq &M_\alpha=(D_1\times \cdots \times D_k)\times (D_{k+1}\times \cdots \times D_s)\\
&\subseteq& (R\cap M_\alpha)\times (S\cap M_\alpha)=R_\alpha\times S_\alpha.
\end{eqnarray*}
Hence $M_\alpha=R_\alpha\times S_\alpha$, $R_\alpha,S_\alpha\neq 1$ and $R_\alpha,S_\alpha$ are minimal normal subgroups of $H_\alpha$. In particular, from Claim~$2$ the lemma is proved.

Finally  we may assume that, for every $j\in \{1,\ldots,l\}$, $\pi_{S,j,\alpha}$ is not surjective.  Let $\pi_{R,\alpha}:M_\alpha\to R$ and $\pi_S:M_\alpha\to S$ be the projections of $M_\alpha$ on $R$ and $S$. Let $K_R$ and $K_S$ be the kernels of $\pi_{R,\alpha}$ and $\pi_{S,\alpha}$ respectively. As $\pi_{R,j,\alpha}$ is surjective for every $j$, the group $M_\alpha/K_R$ is isomorphic to a direct product of $s\geq 1$ copies of $T$, that is, every composition factor of $M_\alpha/K_R$ is isomorphic to $T$. As, for all $j$, $\pi_{S,j,\alpha}$ is not surjective, the group $M_\alpha/K_S$ has no composition factor isomorphic to $T$. Since $M_{\alpha}/(K_RK_S)$ is a homomorphic image of $M_\alpha/K_R$ and of $M_\alpha/K_S$, we obtain that $M_\alpha/(K_RK_S)=1$ and $M_{\alpha}=K_RK_S=K_R\times K_S$. It follows that $K_S=M_\alpha\cap R=R_\alpha$ and $K_R=M_\alpha\cap S=S_\alpha$. Therefore $M_\alpha=R_\alpha\times S_\alpha$. Furthermore, as $R_\alpha\cong M_\alpha/K_R\cong T^s$ and
  $H_\alpha$ acts transitively on $\{T_{R,j}\}_j$, we see that $R_\alpha$ is a minimal normal subgroup of $H_\alpha$. In particular, from Claim~$2$ the lemma is proved.
\end{proof}

\begin{notation}\label{not:4}{\rm From Lemma~\ref{lemma:projRSsurj}, we may now assume that, for every $j\in \{1,\ldots,l\}$ and for every $\alpha\in \Delta$, $\pi_{R,j,\alpha}$ and $\pi_{S,j,\alpha}$ are not surjective. Let $\alpha\in\Delta$. 
Let $U_{R,j}=\pi_{R,j,\alpha}(M_{\alpha})$, $U_{S,j}=\pi_{S,j,{\alpha}}(M_{\alpha})$ and define $U_R=U_{R,1}\times \cdots\times U_{R,l}$, $U_S=U_{S,1}\times \cdots \times U_{S,l}$. By construction, $M_{\alpha}$ projects surjectively on each of the $2l$ direct factors of $U_R\times U_S$, that is, $M_{\alpha}$ is a subdirect subgroup of $\prod_{j}U_{R,j}\times \prod_{j}U_{S,j}$. Let $\Sigma_R$ be the set of right cosets of $U_R$ in $R$, which we denote by $\Sigma_R=R/U_R$. Similarly, let $\Sigma_S=S/U_S$ be the set of right cosets of $U_S$ in $S$. Since $U_R=\prod_j U_{R,j}$ and $U_S=\prod_{j}U_{S,j}$, the $R$-action on  $\Sigma_R$ and the $S$-action on $\Sigma_S$ admit \emph{cartesian decompositions}, namely $\Sigma_R=\prod_{j=1}^l(T_{R,j}/U_{R,j})$ and  $\Sigma_S=\prod_{j=1}^l(T_{S,j}/U_{S,j})$. By construction, $M=R\times S$ acts transitively and faithfully with product action on $\Sigma=\Sigma_R\times\Sigma_S$.

 We claim that $H_{\alpha}$ normalises $U_R$ and $U_S$. In fact, since $H_{\alpha}$ normalises $M_{\alpha}$ and acts transitively on $\{T_{R,j}\}_j$ and on $\{T_{S,j}\}_j$, we get that $H_{\alpha}$ acts transitively on $\{U_{R,j}\}_j$ and on $\{U_{S,j}\}_j$. In particular, $H_{\alpha}$ normalises $\prod_{j=1}^lU_{R,j}=U_R$ and $\prod_{j=1}^lU_{S,j}=U_S$ proving the claim. By transitivity, $|T_{R,j}:U_{R,j}|$ does not depend on $j$, and $|T_{S,j}:U_{S,j}|$ does not depend on $j$, that is, $\Sigma_R$ and $\Sigma_S$ are \emph{homogeneous cartesian decompositions} (a cartesian decomposition $\Lambda_1\times \cdots \times \Lambda_l$ is said to be homogeneous if $|\Lambda_i|=|\Lambda_j|$ for every $i,j\in \{1,\ldots,l\}$). Furthermore, since $H_\alpha$ normalises $U_R$, we have that $U_RH_\alpha$ is a subgroup of $H$ and hence $H$ preserves the cartesian decomposition $\Sigma_R$. Similarly, $H$ preserves the cartesian decomposition $\Sigma_S$. Since $R$ and $S$ are the only minimal
  normal subgroups of $H$ and $M$ acts faithfully on $\Sigma=\Sigma_R\times \Sigma_S$, the group $H$ acts faithfully with the natural product action on $\Sigma=\Sigma_R\times\Sigma_S$. 

Let $H_R$ and $H_S$ be the permutation groups induced by $H$ on $\Sigma_R$ and on $\Sigma_S$ respectively. So $H\subseteq H_R\times H_S\subseteq\Sym(\Sigma_R)\times \Sym(\Sigma_S)\subseteq \Sym(\Sigma)$. Since $R$ is the only minimal normal subgroup of $H_R$ and $R$ is transitive on $\Sigma_R$, we obtain that $\soc(H_R)=R$ and $H_R$ is a quasiprimitive group on $\Sigma_R$. Since $\soc(H_R)$ is the unique minimal normal subgroup of $H_R$ and $\pi_{R,j,\alpha}(M_\alpha)=U_{R,1}<T$, the quasiprimitive group $H_R$ is of type TW, AS or PA. The group $H_R$ is of type TW if $U_R=1$ and $l>1$ (in particular, $R$ acts regularly on $\Sigma_R$), of type AS if $l=1$, and of type PA if $U_R>1$ and $l>1$. Similarly, $H_S$ is a quasiprimitive group of type TW, AS or PA on $\Sigma_S$.}
\end{notation}
We recall the following result (for a proof see~\cite[proof of Theorem~$1$ Case~$2(b)$]{P1}).

\begin{theorem}\label{theo:cartesianA}
Let $G$ be a permutation group on a finite set $\Omega$ preserving a homogeneous cartesian decomposition $\Lambda_1\times \cdots\times\Lambda_l$ of $\Omega$.  Then there is a permutational isomorphism that maps $G$ to a subgroup of $\Sym(\Lambda_1)\wr \Sym(l)$ in its natural product action on $\Lambda_1^l$ and that maps $\Omega$ to the natural cartesian decomposition $\Lambda_1^l$.
\end{theorem}

To state the next result we need a standard  definition. Assume that $G$ is a subgroup of $\Sym(\Lambda)\wr \Sym(l)$ in its natural product action on $\Lambda^l$. Recall that each element of $G$ is of the form $f\sigma$, where $f\in \Sym(\Lambda)^l$ and $\sigma\in \Sym(l)$. We define the $j$th \emph{component} of $G$ as the permutation group induced by $\{f\sigma\in G\mid j\sigma=j\}$ on the $j$th coordinate of $\Lambda^l$. (For a proof of Theorem~\ref{theo:cartesianB} see~\cite[$(2.2)$]{Kov}.) 

\begin{theorem}\label{theo:cartesianB}
Suppose that $G\leq \Sym(\Lambda)\wr\Sym(l)$ is transitive in its product action on $\Lambda^l$. Then there exists an element $x$ in the base group $\Sym(\Lambda)^l$ and a transitive subgroup $K$ of $\Sym(\Lambda)$ such that the $j$th component of $x^{-1}Gx$ is $K$ for every $j\in \{1,\ldots,l\}$. In particular, $G^x\leq K\wr \Sym(l)$.
\end{theorem}

(A proof of Theorems~\ref{theo:cartesianA} and~\ref{theo:cartesianB} can also be found in~\cite{PrSn}.) Our next step is to replace, if necessary, the group $H$ by a suitable conjugate to obtain a simpler form for the action of $H$ on $\Sigma$ (see Notation~\ref{not:3}).  

\begin{notation}\label{not:5}{\rm Assume Notation~\ref{not:11},~\ref{not:3} and~\ref{not:4}. Applying Theorem~\ref{theo:cartesianA} and~\ref{theo:cartesianB} to $H_R$ and $H_S$ separately, we obtain that, up to replacing $H_R$ and $\Sigma_R$, $H_S$ and $\Sigma_S$, by suitable conjugates, we may assume that $\Sigma_R=\Lambda_R^l$, $\Sigma_S=\Lambda_S^l$ (for some sets $\Lambda_R$ and $\Lambda_S$), $H_R\subseteq K_R\wr\Sym(l)$ and $H_S\subseteq K_S\wr\Sym(l)$ in their natural product actions on $\Lambda_R^l$ and $\Lambda_S^l$, where $K_R$ is a transitive subgroup of $\Sym(\Lambda_R)$ and $K_S$ is a transitive subgroup of $\Sym(\Lambda_S)$. Furthermore, since for each $j\in\{1,\ldots,l\}$ the group $T_{R,j}$ is a normal transitive subgroup of the $j$th component of $H_R$ in its action on $\Lambda_R$,  we may assume that $K_R$ is almost simple with socle $T=T_{R,j}$. Similarly, we may assume that $K_S$ is almost simple with socle $T$. Let $\pi_{R,j}:H_{R,j}=N_{H}(T_{R,j})\to K_R$
  and $\pi_{S,j}:H_{S,j}=N_{H}(T_{S,j})\to K_S$ be the natural projections. From Theorem~\ref{theo:cartesianB}, we may assume that $\pi_{R,j}$ and $\pi_{S,j}$ are surjective for each $j$. Summing up, there is an $H$-invariant partition $\Sigma$ of $\Delta$ such that  
$$H\subseteq H_R\times H_S\subseteq (K_R\wr \Sym(l))\times (K_S\wr \Sym(l))$$
and the faithful action of $H$ on $\Sigma$ is the natural product action on $\Lambda_R^l\times \Lambda_S^l$. In particular, the elements of $H$ can be written as $h=(k_{R,1},\ldots,k_{R,l},k_{S,1},\ldots,k_{S,l})s_Rs_S$ with $k_{R,j}\in K_R$ and $k_{S,j}\in K_S$ for each $j$,  $s_R$ a permutation of the $l$ labels $\{(R,j)\}_j$ and $s_S$ a permutation of the $l$ labels $\{(S,j)\}_j$. 

Fix $\lambda_R$ an element of $\Lambda_R$ and $\lambda_S$ an element of $\Lambda_S$. We denote by $\sigma_R$ the element $(\lambda_R,\ldots,\lambda_R)$ of $\Sigma_R=\Lambda_R^l$ and by $\sigma_S$ the element $(\lambda_S,\ldots,\lambda_S)$ of $\Sigma_S=\Lambda_S^l$. Also, we fix $\alpha_0$ a vertex of $V\Gamma^\Delta=\Delta$ with $\alpha_0$ lying in the block $(\sigma_R,\sigma_S)$ of $\Sigma$. }
\end{notation}

\begin{theorem}\label{prop:PAtype2qp}
Assume
Notations~\ref{not:distance2},~\ref{not:3},~\ref{not:4} and~\ref{not:5}. Then $l$ is less than or equal to $(d(d-1))^{d(d-1)}\min\{|T_{\lambda_R}|,|T_{\lambda_S}|\}^{2d(d-1)}$. Furthermore, $(\Gamma,G)$ uniquely determines two elements $(\Lambda_R,T)$, $(\Lambda_S,T)$ in $\mathcal{A}(d(d-1))$ with the stabilisers of the vertices of $\Lambda_R$ conjugate to $T_{\lambda_R}$ and with the stabilisers of the vertices of $\Lambda_S$ conjugate to $T_{\lambda_S}$. 
\end{theorem}

\begin{proof}
Let $\Gamma_\Sigma$ be the quotient graph of $\Gamma^\Delta$ corresponding to the partition $\Sigma=\Lambda_R^l\times \Lambda_S^l$ of $V\Gamma^\Delta=\Delta$ and let $\Gamma_\Sigma((\sigma_R,\sigma_S))$ denote  the set of neighbours of $(\sigma_R,\sigma_S)$ in $\Gamma_\Sigma$.
Let $(\eta_R^i,\eta_S^i)=(\lambda_{R,1}^i,\ldots,\lambda_{R,l}^i,\lambda_{S,1}^i,\ldots,\lambda_{S,l}^i)$, for $1\leq i\leq s$, be representatives of the orbits of $M_{(\sigma_R,\sigma_S)}$ in the action on $\Gamma_{\Sigma}((\sigma_R,\sigma_S))$. Since $M$ is transitive on $\Delta$ and $\Sigma=\Sigma_R\times\Sigma_S$ is a system of imprimitivity for $H$ acting on $\Delta$, we obtain that $M_{(\nu_R,\nu_S)}$ is transitive on $(\nu_R,\nu_S)$ for every $(\nu_R,\nu_S)\in \Sigma$. So Lemma~\ref{lemma:nrorbits} applies to $H$, $M$ and $\Sigma$, and we have $s\leq d(d-1)$. Since $M_{(\sigma_R,\sigma_S)}=R_{\sigma_R}\times S_{\sigma_S}=(T_{\lambda_R})^l\times (T_{\lambda_S})^l$, we get
\begin{eqnarray}\label{eq:111RS}
\Gamma_{\Sigma}((\sigma_R,\sigma_S))&=&
\bigcup_{i=1}^s(\eta_R^i,\eta_S^i)^{M_{(\sigma_R,\sigma_S)}}\bigcup_{i=1}^s
\left(
(\eta_R^i)^{R_{\sigma_R}}\times (\eta_S^i)^{S_{\sigma_S}}
\right)
\\\nonumber
&=&\bigcup_{i=1}^s\left(
(\lambda_{R,1}^i)^{T_{\lambda_R}}\times \cdots \times(\lambda_{S,l}^i)^{T_{\lambda_S}} \right).
\end{eqnarray}
Now we prove five claims from which the theorem will follow.

\smallskip

\noindent\textsc{Claim~$1_R$. }$\bigcup_{i=1}^{s}(\lambda_{R,j}^i)^{K_{\lambda_R}}=\bigcup_{i=1}^s(\lambda_{R,j}^i)^{T_{\lambda_R}}$ for every $j\in \{1,\ldots,l\}$.

\noindent As $T_{\lambda_R}\subseteq K_{\lambda_R}$, the right hand side is contained in the left hand side. Fix $i$ in $\{1,\ldots,s\}$ and $k$ in $K_{\lambda_R}$. From Lemma~\ref{lemma:proj}, we have that $\pi_{R,j}(H_{R,j}\cap H_{(\sigma_R,\sigma_S)})=K_{\lambda_R}$ for each $j\in \{1,\ldots,l\}$. Hence there exists $h=(k_{R,1}',\ldots,k_{S,l}')s\in H_{R,j}\cap H_{(\sigma_R,\sigma_S)}$ with  $k'_{R,j}=k$. We obtain that  $(\eta_R^i,\eta_S^i)^{h}\in \Gamma_{\Sigma}((\sigma_R,\sigma_S))$. Since $(R,j)s=(R,j)$, the $(R,j)$th coordinate of $(\eta_R^i,\eta_S^i)^{h}$ is
$$
(\lambda_{(R,j)s^{-1}}^{i})^{k'_{(R,j)s^{-1}}}
=(\lambda_{R,j}^i)^{k'_{R,j}}
=(\lambda_{R,j}^i)^{k}.
$$
So, from~$(\ref{eq:111RS})$, we have $(\lambda_{R,j}^i)^{k}\in \bigcup_{i=1}^s(\lambda_{R,j}^i)^{T_{\lambda_R}}$ and Claim~$1_R$ follows.~$_\blacksquare$

\smallskip

\noindent\textsc{Claim~$1_S$. }$\bigcup_{i=1}^{s}(\lambda_{S,j}^i)^{K_{\lambda_S}}=\bigcup_{i=1}^s(\lambda_{S,j}^i)^{T_{\lambda_S}}$ for every $j\in \{1,\ldots,l\}$.

\noindent The proof of this claim is similar to the proof of Claim~$1_R$.~$_\blacksquare$

\smallskip

\noindent\textsc{Claim~$2_R$ }
$
\bigcup_{i=1}^s(\lambda_{R,j}^i)^{T_{\lambda_R}}=\bigcup_{i=1}^s(\lambda_{R,1}^i)^{T_{\lambda_R}}$ for every $j\in \{1,\ldots,l\}$.

\noindent Fix $j$ in $\{1,\ldots,l\}$. Since the left hand side and the right hand side are $T_{\lambda_R}$-invariant, it suffices to show that  $\lambda_{R,j}^v\in \bigcup_{i=1}^s(\lambda_{R,1}^i)^{T_{\lambda_R}}$ for every $v\in\{1,\ldots,s\}$. (This would prove that the left hand side is contained in the right hand side and the same argument proves the reverse inclusion.) Fix $v$ in $\{1,\ldots,s\}$. Recall that $H_{\alpha_0}$ is transitive on $\{T_{R,j}\}_j$ and hence so is $H_{(\sigma_R,\sigma_S)}$. Therefore there exists $$h=(k_{R,1},\ldots,k_{S,l})s\in H_{(\sigma_R,\sigma_S)} \quad\textrm{with }(R,1)s^{-1}=(R,j).
$$
We obtain that $(\eta_R^v,\eta_S^v)^h\in \Gamma_{\Sigma}((\sigma_R,\sigma_S))$. Thus

\begin{eqnarray*}
(\eta_R^v,\eta_S^v)^h&=&(\lambda_{R,1}^v,\ldots,\lambda_{S,l}^v)^{
(k_{R,1},\ldots,k_{S,l})s}\\\nonumber&=&(
(\lambda_{(R,1)s^{-1}}^v)^{k_{(R,1)s^{-1}}},\ldots,
(\lambda_{(S,l)s^{-1}}^v)^{k_{(S,l)s^{-1}}}).
\end{eqnarray*}
Recalling that $(R,1)s^{-1}=(R,j)$, we see that the $(R,1)$th entry of $(\eta_R^v,\eta_S^v)^{h}$ is $(\lambda_{R,j}^v)^{k_{R,j}}$. Therefore from~($\ref{eq:111RS}$) we obtain that 
$$(\lambda_{R,j}^v)^{k_{R,j}}\in \bigcup_{i=1}^s(\lambda_{R,1}^i)^{T_{\lambda_R}}.$$
As $k_{R,j}\in K_{\lambda_R}$, Claim~$1_R$ yields 
$\lambda_{R,j}^v\in \bigcup_{i=1}^s(\lambda_{R,1}^i)^{T_{\lambda_R}}$
and Claim~$2_R$ follows.~$_\blacksquare$ 

\smallskip

\noindent\textsc{Claim~$2_S$. }$\bigcup_{i=1}^{s}(\lambda_{S,j}^i)^{T_{\lambda_S}}=\bigcup_{i=1}^s(\lambda_{S,1}^i)^{T_{\lambda_S}}$ for every $j\in \{1,\ldots,l\}$.

\noindent The proof of this claim is similar to the proof of Claim~$2_R$.~$_\blacksquare$

\smallskip

Claim~$2_R$ yields that the $ls$ elements $\{\lambda_{R,j}^i\}_{(R,j),i}$ are in at most $s$ distinct $T_{\lambda_R}$-orbits. In particular, for each $i\in \{1,\ldots,s\}$, the $R_{\sigma_R}$-orbit $(\eta_R^i)^{R_{\sigma_R}}$ contains an element with at most $s$ distinct entries from $\Lambda_R^l$. Therefore, replacing $\eta_R^i$ with a suitable element from $(\eta_R^i)^{R_{\sigma_R}}$ if necessary, we may assume that there are at most $s$ distinct elements of $\Lambda_R$ among the entries of $\eta_R^i$. A similar argument applies for $\eta_S^i$ and    we may assume that there are at most $s$ distinct elements of $\Lambda_S$ among the entries of $\eta_S^i$. 

Since $M=R\times S$ is transitive on $V\Gamma^\Delta$, we may choose $r_i\in R$ and $s_i\in S$ such that $(\sigma_R,\sigma_S)^{r_is_i}=(\eta_R^i,\eta_S^i)$. Note that as $\eta_R^i$ has at most $s$ distinct entries from $\Lambda_R^l$, $\eta_S^i$ has at most $s$ distinct entries from $\Lambda_S^l$ and all the entries of $\sigma_R$ and of $\sigma_S$ are equal, the elements $r_i\in R$ and $s_i\in S$ can be chosen so that their $l$ coordinates contain at most $s$ distinct entries from $T$.
 
For each $\beta\in \Gamma^\Delta(\alpha_0)$, let $r_\beta$ be an element in $R$ and $s_\beta$ an element of $S$ with $\beta=\alpha_0^{r_\beta s_\beta}$. 

\smallskip

\noindent\textsc{Claim~$3$. }$T_R^l=\langle r_\beta\mid \beta\in \Gamma^\Delta(\alpha_0)\rangle T_{\lambda_R}^l$ and $T_S^l=\langle s_\beta\mid \beta\in \Gamma^\Delta(\alpha_0)\rangle T_{\lambda_S}^l$.

\noindent 
Set $U=\langle r_\beta s_\beta\mid \beta \in \Gamma^\Delta(\alpha_0)\rangle$ and let $\overline{\Gamma}$ be the subgraph of $\Gamma^\Delta$ induced on the set $\alpha_0^U$. We show that $\Gamma^\Delta=\overline{\Gamma}$. By the definitions of $U$ and $\overline{\Gamma}$, we have $\alpha_0\in V\overline{\Gamma}$ and  $\Gamma(\alpha_0)\subseteq V\overline{\Gamma}$. Therefore, since $\overline{\Gamma}$ is $U$-vertex-transitive, every vertex of $\overline{\Gamma}$ has valency $|\overline{\Gamma}(\alpha_0)|=|\Gamma^\Delta(\alpha_0)|$. Since $\Gamma^\Delta$ is connected, this yields $\Gamma^\Delta=\overline{\Gamma}$. In particular, $U$ acts transitively on $V\Gamma^\Delta$ and so 
$M=UM_{\alpha_0}$. As $M_{\alpha_0}$ is a subgroup of $M_{(\sigma_R,\sigma_S)}$, we have 
$M=UM_{(\sigma_R,\sigma_S)}$ and
\begin{eqnarray*}
T_R^l\times T_S^l&=&M=\langle r_\beta s_\beta\mid \beta \in \Gamma^\Delta(\alpha_0)\rangle M_{(\sigma_R,\sigma_S)}\\
&\subseteq& 
\left(\langle r_\beta\mid \beta\in \Gamma^\Delta(\alpha_0)\rangle T_{\lambda_R}^l\right)
\times
\left(\langle s_\beta\mid \beta\in \Gamma^\Delta(\alpha_0)\rangle T_{\lambda_S}^l\right).
\end{eqnarray*}
The claim follows.~$_\blacksquare$ 

\smallskip

Fix $\beta$ in $\Gamma^\Delta(\alpha_0)$. Since $\beta=\alpha_0^{r_\beta s_\beta}\in\Gamma^\Delta(\alpha_0)$, we get $(\sigma_R,\sigma_S)^{r_\beta s_\beta}\in \Gamma_\Sigma((\sigma_R,\sigma_S))$ and hence $(\sigma_R,\sigma_S)^{r_\beta s_\beta}\in (\eta_R^{i_\beta},\eta_S^{i_\beta})^{M_{(\sigma_R,\sigma_S)}}$ for some $i_\beta\in \{1,\ldots, s\}$. In particular, as $M_{(\sigma_R,\sigma_S)}=R_{\sigma_R}\times S_{\sigma_S}$, there exists $z_{R,\beta}\in R_{\sigma_R}$, $z_{S,\beta}\in S_{\sigma_S}$ such that 
$$
(\sigma_R,\sigma_R)^{r_\beta s_\beta z_{R,\beta} z_{S,\beta}}=(\eta_R^{i_\beta},\eta_S^{i_\beta}).
$$
Since by the definitions of the $r_i$ and $s_i$ we have $(\sigma_R,\sigma_S)^{r_is_i}=(\eta_R^i,\eta_S^i)$, we obtain $$(\sigma_R,\sigma_S)^{(r_\beta z_{R,\beta}r_{i_\beta}^{-1})(s_\beta z_{S,\beta}s_{i_\beta}^{-1})}=(\sigma_R,\sigma_S),$$ that is to say, $y_{R,\beta}=r_\beta z_{R,\beta}r_{i_\beta}^{-1}\in R_{\sigma_R}$ and $y_{S,\beta}=s_\beta z_{S,\beta}s_{i_\beta}^{-1}\in S_{\sigma_S}$. Therefore, for every $\beta\in \Gamma^\Delta(\alpha_0)$, there exists $i_\beta\in \{1,\ldots,s\}$, $y_{R,\beta},z_{R,\beta}\in R_{\sigma_R}$ and $y_{S,\beta},z_{S,\beta}\in S_{\sigma_S}$ such that 

$$r_\beta=y_{R,\beta}r_{i_\beta}(z_{R,\beta})^{-1}\qquad \textrm{and}\qquad s_\beta=y_{S,\beta}s_{i_\beta}(z_{S,\beta})^{-1}.$$
Now Lemma~\ref{lemma:auxiliary} and Claim~$3$ applied to $\{r_\beta\}_\beta$ and $\{s_\beta\}_\beta$ imply that $l$ is less than or equal to $d(d-1)^{d(d-1)}\min\{|T_{\lambda_R}|,|T_{\lambda_S}|\}^{2d(d-1)}$.

Set $M_{R,j}=T_{R,1}\times \cdots \times T_{R,j-1}\times T_{R,j+1}\times \cdots \times T_{R,l}\times S$ and $M_{S,j}=R\times T_{S,1}\times \cdots \times T_{S,j-1}\times T_{S,j+1}\times \cdots \times T_{S,l}$ for $j\in \{1,\ldots,l\}$. Since $M_{R,j}$ is a maximal normal subgroup of $M$, we obtain that either $M_{R,j}$ is transitive on $V\Gamma^\Delta$ or $M_{R,j}$ is $1$-closed in $M$. As $M_{\alpha_0}$ is a subdirect subgroup of $M_{(\sigma_R,\sigma_S)}$, we have $M_{\alpha_0}M_{R,j}=M_{(\sigma_R,\sigma_S)}M_{R,j}=T_{\lambda_R}\times M_{R,j}<M$ and so $M_{R,j}$ is $1$-closed in $M$. Similarly, $M_{\alpha_0}M_{S,j}=T_{\lambda_S}\times M_{S,j}$ and $M_{S,j}$ is $1$-closed in $M$. Therefore, by Definition~\ref{def:nq}, the pairs $(\Gamma^\Delta_{M_{R,j}},M/M_{R,j})$ and  $(\Gamma^\Delta_{M_{S,j}},M/M_{S,j})$ lie in $\mathcal{A}(d(d-1))$ for each $j\in \{1,\ldots,l\}$. Furthermore, since $H$ acts transitively on $\{M_{R,j}\}_j$ and on $\{M_{S,j}\}_j$, we obtain that $\Gamma^\Delta_{M_{R,j}}\cong \Gamma^\Delta_{M_{R,i}}$ and $\Gamma^\Delta_{M_{S,j}}\cong \Gamma^\Delta_{M_{S,i}}$ for every $i,j\in \{1,\ldots,l\}$. This shows that $(\Gamma,G)$ uniquely determines the elements $(\Gamma^\Delta_{M_{R,1}},M/M_{R,1})$ and  $(\Gamma^\Delta_{M_{S,1}},M/M_{S,1})$ of $\mathcal{A}(d(d-1))$. Finally, the stabiliser in $M/M_{R,1}$ of the vertex $\alpha_0^{M_{R,1}}$ of $\Gamma^\Delta_{M_{R,1}}$ is $M_{\alpha_0}M_{R,1}/M_{R,1}\cong T_{\lambda_R}$. A similar argument holds for $S$ and the theorem is proved.
\end{proof}

Now we are ready to prove Theorem~\ref{thm:mainbiqp}.

\smallskip

\noindent\emph{Proof of Theorem~\ref{thm:mainbiqp}. }If $G_{(\alpha,1)}$ is transitive on $\Delta\times\{2\}$, then from Lemma~\ref{lemma:silly} the graph $(\Gamma,G)$ is $f$-bounded for $f(d)=d!(d-1)!$ and Part~$(1)$ of Theorem~\ref{thm:mainbiqp} holds for $(\Gamma,G)$. So assume that this is not the case. Then $G$ satisfies Part~$(A)$ or~$(B)$ of Theorem~\ref{thm:cheryl}.   If $G$ is of type~$(A)$ or of type~$(B)$~$(ii)$, then the result follows from Theorem~\ref{thm:Bii}. Assume that $G$ of type~$(B)$~$(i)$. We use Notations~\ref{not:distance2},~\ref{not:3},~\ref{not:4} and~\ref{not:5}. If $R$ is abelian or if $\pi_{R,j,\alpha}$ or $\pi_{S,j,\alpha}$ are surjective for some $j$ and for some $\alpha$, then from Lemmas~\ref{lemma:Rabelian} and~\ref{lemma:projRSsurj} the graph $(\Gamma,G)$ is $f$-bounded for $f$ as in Part~$(1)$ of Theorem~\ref{thm:mainbiqp}. 
Finally, assume that $R$ is non-abelian and $\pi_{R,j,\alpha},\pi_{S,j,\alpha}$ are not surjective. Let $(\Lambda_R,T)$ and $(\Lambda_S,T)$ be as in Theorem~\ref{prop:PAtype2qp}. We have $l\leq d_0^{d_0}\min\{|T_{\lambda_R}|,|T_{\lambda_S}|\}^{2d_0}$ with $d_0=d(d-1)$ and with $\lambda_R\in V\Lambda_R,\lambda_S\in V\Lambda_S$. Assume that $(\Lambda_R,T)$ is $g_R$-bounded for some $g_R$ and $(\Lambda_S,T)$ is $g_S$-bounded for some $g_S$. Then, 
$|M_{\alpha_0}|\leq 
|M_{(\sigma_R,\sigma_S)}|= 
(|T_{\lambda_R}||T_{\lambda_S}|)^l\leq 
(g_R(d_0)g_S(d_0))^l$ with $d_0=d(d-1)$. So $(\Gamma^\Delta,M)$ is $f'$-bounded for $f'(d_0)=(g_R(d_0)g_S(d_0)^{d_0^{d_0}\min\{|T_{\lambda_R}|,|T_{\lambda_S}|\}^{2d_0}}$. Theorem~\ref{thm:1} with $f_3(d)$ as in Remark~\ref{rmrm} yields that $(\Gamma^{\Delta},H)$ is $f$-bounded for $f(d_0)=(d_0f'(d_0))!$. From Notation~\ref{not:distance2}, 
\begin{eqnarray*}
|G_{(\alpha,1)}|&=&|H_{\alpha}|\leq f(d_0)=(d_0f'(d_0))!\\
&=&(d_0(g_R(d_0)g_S(d_0))^{d_0^{d_0}\min\{g_R(d_0),g_S(d_0)\}^{2d_0}})!=\overline{g_R\ast g_S}(d)
\end{eqnarray*} and hence $(\Gamma,G)$ is $\overline{g_R\ast g_S}$-bounded. Conversely, assume that $(\Gamma,G)$ is $f$-bounded for some $f$. By Lemma~\ref{tedious}, $(\Gamma^\Delta,H)$ is $\tilde{f}$-bounded. As $M_{\alpha_0}$ is a subdirect subgroup of $M_{(\sigma_R,\sigma_S)}\cong T_{\lambda_R}^l\times T_{\lambda_S}^l$, we have 
$|T_{\lambda_R}|\leq |M_{\alpha_0}|\leq |H_{\alpha_0}|$ and similarly $|T_{\lambda_S}|\leq |H_{\alpha_0}|$. So $(\Lambda_R,T)$ and $(\Lambda_S,T)$ are $\tilde{f}$-bounded.~$\qed$ 

\section{Examples}\label{sec:examples}
\begin{example}\label{ex:1}{\rm There are many natural examples of $(\Gamma,G)\in \mathcal{A}(d)$ admitting a $1$-closed subgroup $N$ where $(\Gamma_N,G/N)$ is $f$-bounded and $N_\alpha$ is not bounded by a function of $d$. For instance, let $X$ and $Y$ be connected vertex-transitive graphs of valency $d_X$ and $d_Y$, respectively. We recall that the \emph{lexicographic product} $X[Y]$ of $X$ and $Y$ is the graph with vertex set $VX\times VY$ where $(x,y)$ is adjacent to $(x',y')$ if and only if $x,x'$ are adjacent in $X$ or $x=x'$ and $y,y'$ are adjacent in $Y$. Note that $X[Y]$ is connected of valency $d=d_Y+d_X|VY|$. Clearly, the wreath product $G=\Aut(Y)\wr\Aut(X)$ acts vertex-transitively on $X[Y]$. If $N=\Aut(Y)^{VX}$ is the base group of $G$, then $G/N\cong \Aut(X)$, the normal quotient $X[Y]_N$ is isomorphic to $X$ and $N_{(x,y)}=\Aut(Y)_y\times \Aut(Y)^{VX\setminus\{x\}}$. In particular, if $|VX|$ is not bounded by a function of $d$, then $|N_{(x,y)}|$ is not bounded by a function
  of $d$. Also, if $(X,\Aut(X))$ is $f$-bounded, then $(X[Y]_N,G/N)$ is $f$-bounded.

As an explicit example take $X=C_n$ the cycle of length $n$ and $Y=K_2$ the complete graph on two vertices. We have $(\Gamma,G)\in \mathcal{A}(4)$ and $(\Gamma_N,G/N)$ is $2$-bounded because $\Gamma_N\cong X$ and $\Aut(X)$ is the dihedral group of order $2n$. Furthermore, $|N_{(x,y)}|=2^{n-1}$ and 
hence $N_{(x,y)}$ can be exponential in the number of vertices of $\Gamma$ with $G/N$ having stabiliser $C_2$.
}
\end{example}

In Examples~\ref{ex:2},~\ref{ex:3} and~\ref{ex:4}, we use the notation of Theorem~\ref{thm:mainqp}. In each of the examples $G$ is a quasiprimitive group of type PA with socle $T^2$. We denote by $D_n$ the dihedral group of order $n$.

\begin{example}\label{ex:2}{\rm 
In this example we give an infinite family of $(\Gamma,G)$ vertex-primitive and locally quasiprimitive with  $(\Lambda,H)$ not quasiprimitive.

Let $q$ be a prime power $q=p^e\geq 4$ and  $n\geq 3$ with
$\textrm{Gcd}(q^2-1,n)=1$. Let $T$ be the simple group
$\mathrm{PSL}(n,q^2)=\mathrm{SL}(n,q^2)$ and $H=T\rtimes\langle F,\tau\rangle$
where   $F$ is the field automorphism of order $2$ of
$\mathbb{F}_{q^2}$ and $\tau$ is the graph
automorphism, that is, $x^\tau=(x^{-1})^{tr}$. Let $K$ be the group
$C_H(F)=(\mathrm{SL}(n,q)\rtimes \langle \tau\rangle)\times
\langle F\rangle$.  From~\cite{LiK}, we see  that $K$ is maximal in $H$. Let $\Delta$ be
the set of right cosets of $K$ in $H$  and denote by $\delta_0$ the coset $K$ of $\Delta$. So, 
$H$ acts primitively on $\Delta$. Let $\lambda$ be an element of order $q+1$ in
$\mathbb{F}_{q^2}$ and \[
x=\left(
\begin{array}{ccc}
\lambda&0&0\\
0&\lambda^{-1}&0\\
0&0&I_{n-2}\\
\end{array}
\right).
\]
Denote by $\delta_1$ the coset $Kx$ in $\Delta$. We claim that $H_{\delta_0}$
acts faithfully on the suborbit $\delta_1^{H_{\delta_0}}$ of
$H$. By our choice of $x$, the element $F$ does not fix
$\delta_1$. Also, it is easy to find elements of $\mathrm{SL}(n,q)$ not
fixing $\delta_1$. Since $\mathrm{SL}(n,q)$ and $\langle F\rangle$ are the only
minimal normal subgroups of $H_{\delta_0}$, our claim is proved. Now we
claim that $\mathrm{SL}(n,q)$ acts transitively on
$\delta_1^{H_{\delta_0}}$. It is easy to check that the element 
\[
Fy, \qquad \textrm{with }y=\left(
\begin{array}{ccc}
0 &1&0\\
-1&0&0\\
 0&0&I_{n-2}
\end{array}
\right),
\] 
of $H_{\delta_0}$ fixes $\delta_1$. Similarly, $\tau y$ fixes
$\delta_1$. Since $H_{\delta_0}=\mathrm{SL}(n,q)\langle
Fy,\tau y\rangle$, we get that $\mathrm{SL}(n,q)$ is transitive on
$\delta_1^{H_{\delta_0}}$. Finally, a direct computation shows that
$(\delta_0,\delta_1)^{yx}=(\delta_1,\delta_0)$. 
Let $\Lambda$ be the $H$-orbital graph containing
the arc $(\delta_0,\delta_1)$. Since $H$ is primitive, $\Lambda$ is connected.

We have shown that $\Lambda$ is an undirected
$H$-arc-transitive graph, that $H$ acts primitively on
$V\Lambda$ and that $H_{\delta_0}$ acts faithfully on
$\Lambda(\delta_0)$. Also, $\mathrm{SL}(n,q)$ acts transitively
on $\Lambda(\delta_0)$  and $\langle F\rangle$ acts
intransitively and semiregularly on $\Lambda(\delta_0)$. In
particular $(\Lambda,H)$ is 
not locally quasiprimitive. 

Let $W$ be the wreath product $H\wr \Sym(2)$ endowed with the
product action on $\Omega=\Delta^2$. Write $W=(H\times H)\rtimes
\langle\pi\rangle$, 
where $\pi^2=1$ and $(h,1)^\pi=(1,h)$ for $h\in H$. Let $T$ be the
socle of $H$ and $N=T^2$ the socle of $W$. Consider $G=N\langle
(F,\tau),(\tau,F),\pi\rangle$. Note that
each of $(F,\tau),(\tau,F),\pi$ has order $2$ and
$(F,\tau)^\pi=(\tau,F)$, so $G/N\cong D_8$.  The projection of $N_G(T\times 1)=(T\times T)\langle
(F,\tau),(\tau,F)\rangle$ onto the first coordinate is the whole of
$H$. As $G$ contains $\langle N,\pi\rangle$ and as $H$ is primitive on
$\Delta$, we 
obtain that $G$ acts primitively on $\Omega$.

Let $\Gamma$ be the $W$-orbital graph
corresponding to the suborbit 
$\delta_1^{H_{\delta_0}}\times \delta_1^{H_{\delta_0}}$ of
$W_{(\delta_0,\delta_0)}$. Since $G$ is primitive, $\Gamma$ is connected and since $\Lambda$ is undirected, so is $\Gamma$. As $T_{\delta_0}=\mathrm{SL}(n,q)$ is transitive on
$\delta_1^{H_{\delta_0}}$, the group  $N_{(\delta_0,\delta_0)}$ acts
transitively on $\Gamma((\delta_0,\delta_0))$. Therefore the graph $\Gamma$ is
$G$-arc-transitive. 

We claim that $G_{(\delta_0,\delta_0)}$ is quasiprimitive on $\Gamma((\delta_0,\delta_0))$.  We have
$$G_{(\delta_0,\delta_0)}=(T_{\delta_0}\times
T_{\delta_0})\langle(F,\tau),(\tau,F),\pi\rangle.$$ Let $X$ be
a normal non-trivial subgroup of $G_{(\delta_0,\delta_0)}$.
As $T_{\delta_0}$ is simple and $\pi\in G_{(\delta_0,\delta_0)}$,  the group
$N_{(\delta_0,\delta_0)}=T_{\delta_0}\times 
T_{\delta_0}$ is a minimal normal subgroup of $G_{(\delta_0,\delta_0)}$.  If
$X\cap N_{(\delta_0,\delta_0)}\neq 
1$, then by minimality $N_{(\delta_0,\delta_0)}\subseteq X$ and $X$ is transitive on
$\Gamma((\delta_0,\delta_0))$. If $X\cap N_{(\delta_0,\delta_0)}=1$, then $X$ centralises
$N_{(\delta_0,\delta_0)}$. The centraliser of $N_{(\delta_0,\delta_0)}$ in $W_{(\delta_0,\delta_0)}$ has
order $4$ and is
generated by $(F,1),(1,F)$. Since $\langle
(1,F),(F,1)\rangle\cap G=1$, no non-trivial element of $G_{(\delta_0,\delta_0)}$
centralises $N_{(\delta_0,\delta_0)}$. Hence $X=1$, a contradiction. This proves
that $G_{(\delta_0,\delta_0)}$ is  
quasiprimitive on $\Gamma((\delta_0,\delta_0))$.}
\end{example}

\begin{example}\label{ex:3}
{\rm 

In this example we give $(\Gamma,G)$ locally semiprimitive with $(\Lambda,T)$ not semiprimitive.
Let $T$ be the simple group $\SL(3,9)$ and $H=T\rtimes \langle F,\tau\rangle$ where $F$ is the Frobenius automorphism of $T$ and $\tau$ is the automorphism of $T$ defined by $x^\tau=(x^{-1})^{tr}$. Let $C$ be a Singer cycle of $T$. Since $|C|=(9^3-1)/(9-1)=7\cdot 13$, we have $C=\langle x\rangle\times\langle y\rangle$ where $x$ has order $7$ and $y$ has order $13$. The normaliser of $C$ in $T$ is $C\rtimes \langle z\rangle$ for some  $z$ of order $3$. From~\cite{ATLAS}, we get that the normaliser $N$ in $H$ of $C$ is $(\langle x,F\rangle \times \langle y,\tau\rangle)\rtimes \langle z\rangle\cong (D_{14}\times D_{26})\rtimes C_3$ where $C_n$ denotes the cyclic group of order $n$. Denote by $K$ the group $\langle x,F\rangle\times \langle y,\tau\rangle$. From~\cite{ATLAS}, we get that $N$ is the unique proper subgroup of $H$ containing $K$ and $N\cap T= C\rtimes \langle z\rangle$ is the unique proper subgroup of $T$ containing $C$. In particular, $N$ is a maximal subgroup of $H$
  and $N\cap T$ is a maximal subgroup of $T$.

Let $\Delta$ be the set of right cosets of $K$ in $H$ and denote by $\delta_0$ the coset $K$ of $\Delta$.  So, 
$H$ acts quasiprimitively on $\Delta$.  
Let $x$ be an involution of $T$ such that $x^\tau=x^F=x$ and denote by $\delta_1$ the coset $Kx$ in $\Delta$. From~\cite{ATLAS} we see that $H_{\delta_0}$ acts faithfully on the suborbit $\delta_1^{H_{\delta_0}}$ and $(H_{\delta_0})_{\delta_1}=\langle F,\tau\rangle$. In particular, $T_{\delta_0}$ acts regularly on $\delta_1^{H_{\delta_0}}$.

Since $x^2=1$, we get 
$(\delta_0,\delta_1)^{x}=(\delta_1,\delta_0)$. 
Let $\Lambda$ be the $H$-orbital graph containing
the arc $(\delta_0,\delta_1)$. Since $N$ is the unique proper subgroup of $H$ containing $K$ and since $x\notin N$ (because $|N:K|=3$), the graph $\Lambda$ is connected. As $T_{\delta_0}$ is transitive on $\Lambda(\delta_0)$, the group $T$ acts arc-transitively on $\Lambda$.
Furthermore $\langle x,F\rangle$ and $\langle y,\tau\rangle$ are normal intransitive and non semiregular subgroups of $H_{\delta_0}$, so $(\Lambda,H)$ is 
not locally semiprimitive.

Let $W$ be the wreath product $H\wr \Sym(2)$ endowed with the
product action on $\Omega=\Delta^2$. Write $W=(H\times H)\rtimes
\langle\pi\rangle$, 
where $\pi^2=1$ and $(h,1)^\pi=(1,h)$ for $h \in H$. Let $T$ be the
socle of $H$ and $N=T^2$ the socle of $W$. Consider $G=N\langle
(\tau,F),(F,\tau),\pi\rangle$. Note that
each of $(\tau,F)$, $(F,\tau)$, $\pi$ has order $2$ and
$(\tau,F)^\pi=(F,\tau)$, so $G/N\cong D_4$.  The projection of $N_G(T\times 1)=N\langle (\tau,F),(F,\tau)\rangle$ onto the first coordinate is the whole of
$H$. As $G$ contains $\langle N,\pi\rangle$ and as $H$ is quasiprimitive on
$\Delta$, we  obtain that $G$ acts quasiprimitively on $\Omega$.

Let $\Gamma$ be the $W$-orbital graph
corresponding to the suborbit 
$\delta_1^{H_{\delta_0}}\times \delta_1^{H_{\delta_0}}$ of
$W_{(\delta_0,\delta_0)}$. We claim that $\Gamma$ is connected, that is, $G=\langle G_{(\delta_0,\delta_0)},(x,x)\rangle$. We have
$$G_{(\delta_0,\delta_0)}N_{(\delta_0,\delta_0)}\langle (\tau_1,F_2),(F_1,\tau_2),\pi\rangle.$$ Since the only proper subgroup of $T$ containing $T_{\delta_0}=C$ is $C\rtimes \langle z\rangle$ and $x\notin C\rtimes\langle z\rangle$, we get $T=\langle T_{\delta_0},x\rangle$. Hence $N=\langle N_{(\delta_0,\delta_0)},(x,x)\rangle$. Therefore $G=\langle G_{(\delta_0,\delta_0)},(x,x)\rangle$.  Since $\Lambda$ is undirected, so is $\Gamma$. As $T_{\delta_0}$ is regular on $\delta_1^{H_{\delta_0}}$, the group  $N_{(\delta_0,\delta_0)}$ acts
regularly on $\Gamma(\alpha)$. In particular the graph $\Gamma$ is
$G$-arc-transitive. 

We claim that $G_{(\delta_0,\delta_0)}$ is semiprimitive on $\Gamma((\delta_0,\delta_0))$.   Let $L$ be
a normal non-trivial subgroup of $G_{(\delta_0,\delta_0)}$. We have to prove that $L$ is either transitive or semiregular. Assume that $L$ is not semiregular. So without loss of generality we may assume that some non-identity element $l$ of $L$ fixes $\beta=(\delta_1,\delta_1)$. As $N_{(\delta_0,\delta_0)}$ acts regularly on $\Gamma((\delta_0,\delta_0))$, we have $l\in \langle(\tau,F),(F,\tau),\pi\rangle$. The group $LN_{(\delta_0,\delta_0)}/N_{(\delta_0,\delta_0)}$ is a non-trivial normal subgroup of the dihedral group $G_{(\delta_0,\delta_0)}/N_{(\delta_0,\delta_0)}$. Thence $LN_{(\delta_0,\delta_0)}/N_{(\delta_0,\delta_0)}$ contains the center of $G_{(\delta_0,\delta_0)}/N_{(\delta_0,\delta_0)}$, that is, $(\tau F,F\tau)N_{(\delta_0,\delta_0)}\in LN_{(\delta_0,\delta_0)}/N_{(\delta_0,\delta_0)}$. Hence we may assume that $l=(\tau F,F\tau)$. Now, $L$ contains the element
\begin{eqnarray*}
l^{(x,1)}l^{-1}&=&(x^{-1},1)(\tau F,F\tau)(x,1)(F\tau,\tau F)=(x^{-1}\tau FxF\tau,1)\\
&=&(\tau x^{-1}FxF\tau,1)=(\tau Fx^2F\tau,1)=(\tau x^{-2}\tau,1)=(x^{-2},1).
\end{eqnarray*}
Therefore $L$ contains $(x,1)$. A similar computation shows that $L$ contains $(1,x)$, $(y,1)$ and $(1,y)$. Also  $L$ contains $N_{(\delta_0,\delta_0)}$. In particular, $L$ is transitive on $\Gamma((\delta_0,\delta_0))$. 
}
\end{example}

\begin{example}\label{ex:4}{\rm 
This remarkable example is described in detail in~\cite[Example~$16$]{WEISS} and we recall here some significant properties related to the work in this paper. We refer to~\cite{WEISS} for the proofs of our claims.

Let $H=\Sym(10)$, $x=(1,2,3)(4,5,6)(7,8,9)$, $y=(1,4,7)(2,5,8)(3,6,9)$, $z=(2,3)(5,6)(8,9)$, $t=(4,7)(5,8)(6,9)$ and $\iota=(1,10)$. Write $K=\langle x,y,z,t\rangle$. Clearly, $K=\langle x,z\rangle\times \langle y,t\rangle\cong \Sym(3)^2$. Let $\Delta$ be the $H$-set $H/K$ and $\Lambda$ be the orbital graph $(K,K\iota)^H$. The graph $\Lambda$ is  connected, $H$-arc-transitive, vertex-quasiprimitive  and the local action  is the natural product action of $\Sym(3)\times \Sym(3)$ of degree $9$, which is not quasiprimitive.

Let $W$ be the wreath product $H\wr \Sym(2)=(H\times H)\rtimes\langle\pi\rangle$ 
where $\pi^2=1$ and $(h_1,h_2)^\pi=(h_2,h_1)$ for $h_1,h_2\in H$. Let $T$ be the
socle of $H$ and $N=T^2$ the socle of $W$. 
Consider $G=N\rtimes\langle \pi,(\iota,\iota)\rangle$ and the subgroup $L=\langle (x,y),(y,x),(z,t),(t,z),\pi\rangle$ of $G$. The projection of $N_G(T\times 1)=N\langle (\iota,\iota)\rangle$ onto the first coordinate of $H^2$ is the whole of $H$; furthermore, $|L|=72$ and $L$ is isomorphic to $\Sym(3)\wr \Sym(2)$. 
Let $\Omega$ be the $G$-set $G/L$. The group $G$ is quasiprimitive of type PA  with socle $N$ in its action on $\Omega$.

Denote by $\alpha$ the element $L$ of $\Omega$, by $\beta$ the element $L(tz\iota,\iota)$ of $\Omega$ and by $\Gamma$ the $G$-orbital graph $(\alpha,\beta)^G$. The graph $\Gamma$ is  connected, $G$-arc-transitive, vertex-quasiprimitive  and the $G$-local action  is the natural primitive action of $\Sym(3)\wr\Sym(2)$ of degree $9$.  Finally, since the projection of $L\cap (H\times H)$ on the first coordinate is exactly the group $K$, the graph uniquely determined by $(\Gamma,G)$ in Theorem~\ref{thm:mainqp} is $\Lambda$. Therefore in this example we have $(\Gamma,G)$ locally primitive with $(\Lambda,T)$ arc-transitive, but not even locally quasiprimitive.
}
\end{example}

\end{document}